\documentclass[%
 aip,
 amsmath,amssymb,
reprint,%
]{revtex4-2}

\usepackage{graphicx}
\usepackage{amsmath} 
\usepackage{amsthm}
\usepackage{amssymb}
\usepackage[T1]{fontenc}
\usepackage{color}

\usepackage[final,authormarkup=superscript]{changes}
\definechangesauthor[color=magenta]{MB}
\definechangesauthor[color=green]{APP}
\definechangesauthor[color=blue]{AW}
\definechangesauthor[color=red]{DK}
\setdeletedmarkup{\textcolor{red}{\sout{{#1}}}}

\usepackage{hyperref}
\hypersetup{
    colorlinks=true,
    linkcolor=blue,
    filecolor=magenta,      
    urlcolor=cyan,
    pdftitle={Overleaf Example},
    pdfpagemode=FullScreen,
    }

\newtheorem{definition}{Definition}
\newtheorem{theorem}{Theorem}

\newtheorem{remark}{Remark}

\newcommand{\R}{\mathbb{R}}

\newcommand{\dd}{\mathrm{d}}
\newcommand{\sign}{\mathrm{sign}}

\makeindex
\begin{document}
\title{Two-dimensional fractional Brownian motion: Analysis in time and frequency domains}


\author{Micha{\l} Balcerek}%
    \affiliation{Faculty of Pure and Applied Mathematics, Hugo Steinhaus Center, Wroc{\l}aw University of Science and Technology, 50-370 Wrocław, Poland}%
    \email{michal.balcerek@pwr.edu.pl}

\author{Adrian Pacheco-Pozo}
\affiliation{Department of Electrical and Computer Engineering, Colorado State University, Fort Collins, CO 80523, USA}
\affiliation{School of Biomedical and Chemical Engineering, Colorado State University, Fort Collins, CO 80523, USA}

\author{Agnieszka Wy{\l}oma{\'n}ska}
    \affiliation{Faculty of Pure and Applied Mathematics, Hugo Steinhaus Center, Wroc{\l}aw University of Science and Technology, 50-370 Wrocław, Poland}%

\author{Krzysztof Burnecki}
    \affiliation{Faculty of Pure and Applied Mathematics, Hugo Steinhaus Center, Wroc{\l}aw University of Science and Technology, 50-370 Wrocław, Poland}%

\author{Diego Krapf}
\affiliation{Department of Electrical and Computer Engineering, Colorado State University, Fort Collins, CO 80523, USA}
\affiliation{School of Biomedical and Chemical Engineering, Colorado State University, Fort Collins, CO 80523, USA}

\begin{abstract}
This article introduces a novel construction of the two-dimensional fractional Brownian motion (2D fBm) with dependent components. Unlike similar models discussed in the literature, our approach uniquely accommodates the full range of model parameters and explicitly incorporates cross-dependencies and anisotropic scaling through a matrix-valued Hurst operator. We thoroughly analyze the theoretical properties of the proposed causal and well-balanced 2D fBm versions, deriving their auto- and cross-covariance structures in both time and frequency domains. In particular, we present the power spectral density of these processes and their increments. Our analytical findings are validated with numerical simulations. This work provides a comprehensive framework for modeling anomalous diffusion phenomena in multidimensional systems where component interdependencies are crucial. 
\end{abstract}


\maketitle
\onecolumngrid 



Two-dimensional fractional Brownian motion extends the idea of random motion by allowing both long-term memory and interdependence between different directions of movement. Unlike standard models, where the horizontal and vertical components are assumed to fluctuate independently, this process accounts for correlations and directional differences in scaling. 
We introduce these features through correlated underlying noises and the matrix-valued Hurst operator, which both make it possible for one direction to show, for example, slower subdiffusive dynamics while the other exhibits faster, superdiffusive behavior. We also show that, depending on how we choose to construct the process, the cross-correlations between components may be symmetric or asymmetric, leading to different temporal and spectral characteristics. The considered model is especially relevant for systems where anisotropy and interdependence cannot be ignored, such as the motion of particles in complex biological environments, the joint evolution of financial indices, or, \added[id=DK]{simply, two correlated fractional Brownian motions}.

\section{Introduction}
\label{sec:intro}
    
The study of anomalous diffusion phenomena has emerged as a critical area of research in diverse scientific fields, from cellular biophysics to financial mathematics \cite{pccp,krapf2015mechanisms,manzo2015review,shen2017single,balcerek_2025_tele,plerou2000economic,andi21}. While standard Brownian motion adequately describes normal diffusion processes characterized by a mean-square displacement that increases linearly with time \cite{einstein05}, multiple complex systems exhibit a non-linear scaling that demands a more sophisticated mathematical framework \cite{pccp}. 
\added[id=MB]{Fractional Brownian motion (fBm) originates from the foundational work of Kolmogorov in the 1940s, who first formulated a Gaussian process with stationary increments and a power-law covariance structure\cite{kol140}.} It was later formalized by Mandelbrot and van Ness, \added[id=MB]{who introduced the integral representation and the Hurst parameter $H$ as a descriptor of memory and anomalous scaling}\cite{NessMandelbrot}. However, when considering multidimensional systems, the scalar parameter $H$ can be insufficient to capture the complex cross-dependencies and anisotropic scaling behaviors observed in empirical data. Multidimensional fBm or vector fBm \cite{meerschaert1999multidimensional,jeon2010fractional,lavancier2009covariance,amblard2010basic}, thus, represents an advancement in the modeling of such anomalous diffusion processes in multidimensional settings. 

A common approach to modeling multidimensional fBm involves treating each spatial component as an independent random walk. This strategy has been employed in a variety of contexts, including the analysis of diffusion in living cells and complex fluids, where isotropic behavior is often assumed over the observation timescales \cite{sikora_diego,krapf2019spectral,sposini2022towards,munoz2023quantitative}. For instance, models with independent fBm components have been used to characterize anomalous transport of membrane proteins \cite{sikora_diego}, analyze motion changes in single-particle trajectories \cite{munoz2023quantitative}, and assess spectral properties of anomalous diffusion processes \cite{sposini2022towards,munoz2023quantitative}. While this approach is well justified in systems where anisotropies are either negligible or not detectable within the experimental resolution, it becomes insufficient in environments where directional dependence or cross-correlations play a significant role. In such cases, assuming independence between components can obscure critical features of the dynamics and lead to misinterpretations. This motivates the development of models that explicitly incorporate cross-dependencies and anisotropic scaling.
    
While one expects many biophysical systems to be isotropic \cite{broetal09,weron2017ergodicity,sabri2020elucidating}, this is not the case when underlying structures have an inherent orientation. In such systems, a tracer may display a different scaling behavior along each spatial dimension and it is necessary to consider the interdependence among the different components. For example, anisotropic scalings emerge in the dynamics of proteins on the cell surface due to the presence of stress fibers \cite{smith1979anisotropic} and actin filaments \cite{sadegh2017plasma}, ions in the brain \cite{vorisek1997evolution, de2011anisotropic}, macromolecules in the nucleocytoplasm during cell division \cite{pawar2014anisotropic}, and species in rocks \cite{van2004anisotropic}.  Beyond physical systems, financial markets often exhibit non-uniform behavior, prompting the use of multivariate models to capture the joint dynamics of multiple assets or economic indicators \cite{balcerek2025two}. These models help in risk management, option pricing, and forecasting \cite{WEERAWARDENA200621,BIELAK2021102308}.
Multidimensional models that allow dependencies between components have been covered in the literature. Notable examples include multidimensional Brownian motion\cite{balcerek2025two,SACERDOTE2016275,TSEKOV1995175,doi:10.1080/17442508.2024.2315274,Kou_Zhong_2016,doi:10.1137/0145060} and multidimensional Ornstein-Uhlenbeck processes\cite{doi:10.1080/17442508.2024.2315274,PhysRevE.99.062221}. Previous discussions also extended to discrete-time models with dependent components, such as vector autoregressive time series\cite{maraj,MOLLER2001143,MORISHIMA1991697,10.1257/jep.15.4.101}, or multidimensional generalized autoregressive conditional heteroskedasticity (GARCH) models \cite{Ling_McAleer_2003,ZHANG2007288,BOUSSAMA20112331}. 

In this article, we propose the construction of a two-dimensional fractional Brownian motion (2D fBm) by introducing a dependence on the underlying noise in the time representation of the process. This construction explicitly considers cross-dependencies and anisotropic scaling by using a matrix-valued Hurst operator that allows direction-dependent scaling properties. \added[id=AW]{Depending on the construction of the 2D fBm, we consider two versions of this process: causal and well-balanced. The }\added[id=DK]{well-balanced process is time-reversible, a property absent in the causal 2D fBm. As a result, these two formulations exhibit distinct correlation structures: symmetric cross-correlations for the well-balanced case and asymmetric cross-correlations for the causal case.}  
We examine the theoretical properties of the two versions of 2D fBm and derive the auto- and cross-covariance structures of such process and their increments. In addition, we calculate the power spectral density (PSD) of these processes \cite{norton2003fundamentals}, a fundamental measure often used in the characterization of time series across different disciplines \cite{hennig2011nature,krapf2013nonergodicity,dean2016sample,fox2021aging,balcerek2023modelling}. At last, we present numerical simulations that showcase and validate our analytical results. 
Although multidimensional fBm with dependent components has been investigated in the existing literature\cite{amblard2010basic,maraj}, our \added[id=AW]{proposition} has the advantage of having a relatively simple construction, based on the correlated Brownian motions used in its integral representation. The new approach uniquely allows for any range of model parameters, a significant advantage over the previous constructions\cite{amblard2010basic}. Consequently, this simplified construction enables a thorough examination of the process characteristics, particularly with respect to its auto- and cross-covariance structures in both the time and frequency domains.

The remainder of the paper is organized as follows. In \autoref{sec:sec1_model}, we introduce the 2D fBm and discuss its two versions, namely causal and well-balanced. Next, in \autoref{sec:3}, we analyze the covariance structure of 2D fBm and its increments, demonstrating the differences for causal and well-balanced cases. In \autoref{sec:4}, we analyze the covariance structure of 2D fBm and the corresponding increments in the frequency domain. In \autoref{sec:5} we present numerical simulations, while the last section contains conclusions. In the Appendix, we present additional calculations and plots.

    
\section{Model}
\label{sec:sec1_model}
\noindent In this section, we propose two specific constructions of the 2D fBm, which serve as a mathematical tool for modeling anomalous diffusion processes in multidimensional spaces. Unlike the classical one-dimensional case, which relies on scalar parameters to characterize diffusion behavior, we employ a matrix-valued $H$ to capture directional dependencies and cross-dimensional correlations. While we only consider the two-dimensional case, the extension to higher dimensions is straightforward.
First, following Stoev and Taqqu\cite{stoev2006rich}, let us introduce the following function
\begin{align}
    \label{eq:f_pm}
    f_\pm(x; t, \beta) = (t-x)_\pm^\beta - (-x)_\pm^\beta,
\end{align}
where the notation $(x)_+$ denotes the positive part of $x$, i.e., 
\begin{align}
    (x)_+ \equiv \begin{cases} x, & \textrm{if } x>0 ,\\
        0, & \textrm{if } x \leq 0,
    \end{cases} 
\end{align}
and $(x)_-$ denotes the negative part of $x$, i.e., $(x)_- = (-x)_+$. Let us consider the two-dimensional process $\mathbf{X}(t) = [X_1(t), X_2(t)]', t\geq 0$, defined as follows 
\begin{align}
    \begin{cases}
    X_1(t) &= \sigma_1 a_{H_1} \displaystyle\int_{-\infty}^\infty f_+\left(s; t, H_1-1/2\right) \dd \tilde W_1(s),\\
    X_2(t) &= \sigma_2 a_{H_2} \displaystyle\int_{-\infty}^\infty f_+\left(s; t, H_2-1/2\right) \dd \tilde W_2(s).   
    \end{cases}
\label{eq:c_fbm}
\end{align}
A similar alternative $\mathbf{X}^*(t) = [X_1^*(t), X_2^*(t)]', t\geq 0$, utilizing both $f_+$ and $f_-$ is given by
\begin{align}
    \begin{cases}
     X_1^*(t) &= \sigma_1 a^*_{H_1} \displaystyle\int_{-\infty}^\infty \left(f_+\left(s; t, H_1-1/2\right) + f_-(s;t,H_1-1/2)\right) \dd \tilde W_1(s),\\
    X_2^*(t) &= \sigma_2 a^*_{H_2}\displaystyle \int_{-\infty}^\infty \left(f_+\left(s; t, H_2-1/2\right) + f_-(s;t,H_2-1/2)\right) \dd \tilde W_2(s).
    \end{cases}
    \label{eq:wb_fbm}
\end{align}
In both cases, $\dd\tilde  W_1(t),  \dd  \tilde W_2(t)$ are $\rho$-correlated Gaussian noises on the real line, $H_1, H_2 \in(0,1)$, and $\sigma_1, \sigma_2>0$. The constants $a_{H_1}$, $a_{H_2}$, $a^*_{H_1}$, and $a^*_{H_2}$ are chosen in such a way that the processes $X_j(t)$ and $X^*_j(t)$ for $j=1,2$ have variances $\sigma_1^2$ and $\sigma_2^2$ for $t=1$, respectively, that is, $\langle X_j^2(1)\rangle = \langle X^{*2}_j(1)\rangle = \sigma_j^2, \textrm{ for } j=1,2$. {Parameter $\rho$ will be referred to as correlation coefficient of the underlying noise.} Let us note that for both processes $\mathbf{X}(t)$ and $\mathbf{X}^*(t)$, the marginals $X_j(t)$ and  $X_j^*(t)$ ($j = 1,2$) are fractional Brownian motions with corresponding Hurst parameters $H_j$ (cf. \cite{NessMandelbrot} for $\mathbf{X}(t)$ and \cite{stoev2006rich} for $\mathbf{X}^*(t)$). %
The difference between the seemingly similar processes $\mathbf{X}(t)$ and $\mathbf{X}^*(t)$ lies in the cross-dependence structure between their components, specifically between $X_1(t)$ and $X_2(t)$, and between $X_1^*(t)$ and $X_2^*(t)$, respectively, which we discuss in the next section.


\noindent To establish the connection between the processes $\mathbf{X}(t)$ and $\mathbf{X}^*(t)$ and the vector fractional Brownian motion discussed in the literature \cite{amblard2010basic, coeurjolly2010multivariate, pipiras2017long}, let us consider two independent Brownian motions $W_1(t), W_2(t)$ on the real line and define $\tilde W_1(t)$ and $\tilde W_2(t)$ in the following way
\begin{align}
    \begin{cases}
        \tilde W_1(t) = W_1(t)\\
        \tilde W_2(t) = \rho W_1(t) + \sqrt{1-\rho^2} W_2(t).
    \end{cases}
\label{eq:construction}
\end{align}
One can easily show that such a construction leads to $\langle \dd \tilde W_1(t) \dd \tilde W_2(t) \rangle = \rho \dd t$. 
Using the construction from Eq. \eqref{eq:construction}, the processes defined in Eqs. \eqref{eq:c_fbm} and \eqref{eq:wb_fbm} can be represented more concisely, as shown in the following definitions. Note that the adjectives ``causal'' and ``well-balanced'' in the processes' names are already established in the literature. We  follow this convention throughout the article. 
\begin{definition}[Causal 2D fBm]
\label{def:vfbm-causal}
Let $H_1, H_2 \in (0, 1)$ and $|\rho|\leq 1$. Causal two-dimensional fractional Brownian motion $\mathbf{X}(t), t\geq 0,$ is defined via the following integral representation
\begin{align}
    \mathbf{X}(t) = \int_{-\infty}^\infty 
        \begin{bmatrix}
            \sigma_1 a_{H_1} f_+(s;t, H_1 - 1/2) & 0 \\
            0 & \sigma_2 a_{H_2} f_+(s;t, H_2-1/2)
        \end{bmatrix} 
        \begin{bmatrix}
            1 & 0 \\
            \rho & \sqrt{1-\rho^2}
        \end{bmatrix}
        \begin{bmatrix}
            \dd W_1(s)\\
            \dd W_2(s)
        \end{bmatrix},
\end{align}    
where $\sigma_1, \sigma_2>0$. The constants $a_{H_j}, j = 1, 2$ are non-negative, and are chosen in such a way that the variances of the marginals $X_j(t)$ at time $t=1$ are equal to $\sigma_j^2$, i.e., 
\begin{align}
    \label{eq:C_causal}
    a^2_{H_j} =\frac{\Gamma(2H_j+1)\sin(H_j\pi)}{\Gamma^2\left(H_j + \frac{1}{2}\right)}.
\end{align}
\end{definition}
\noindent The process $\mathbf{X}(t)$ is a zero-mean Gaussian process and has a strong connection to the so-called operator fractional Brownian motion\cite{pipiras2017long,MAEJIMA1994139,Gustavo_2011,Gustavo_2018}. Such a relation allows for analyzing the  covariance structure of the process $\mathbf{X}(t)$ (see \autoref{thm:vfbm-cov} in \autoref{sec:3}).
\begin{remark}
Let us notice that the marginals $X_j(t)$, $j=1,2$ of the causal 2D fBm follow the well-known Mandelbrot and van Ness definition\cite{NessMandelbrot} of fBm, i.e., 
\begin{align}
    X_j(t) = \sigma_j a_{H_j} \int_{-\infty}^t 
        \left((t-s)_+^{H_j-\frac{1}{2}} - (-s)_+^{H_j - \frac{1}{2}}\right) \dd \tilde W_j(s), \quad \mathrm{for }\ j =1, 2.
\end{align}    
\end{remark}
\noindent Contrary to the causal construction presented in Definition \autoref{def:vfbm-causal}, the well-balanced case adds to the previous case an anti-causal filtering of a Brownian motion. 
We now introduce such a process, together with the calculation of the normalization constant (\autoref{thm:well-balanced-constant} in Appendix \autoref{App:B}). The covariance structure of well-balanced 2D fBm is provided in \autoref{sec:3}. 
\begin{definition}[Well-balanced 2D fBm]
    \label{def:wb-fbm2d}
   Let $H_1, H_2 \in (0, 1)$ and $|\rho|\leq 1$. Well-balanced two-dimensional fractional Brownian motion $\mathbf{X}^*(t)$ is given by the following representation
{
    \begin{align}
    \mathbf{X^*}(t) = \int_{-\infty}^\infty 
        \begin{bmatrix}
            g_1(s; t)  & 0 \\
            0 & g_2(s; t)
        \end{bmatrix} 
        \begin{bmatrix}
            1 & 0 \\
            \rho & \sqrt{1-\rho^2}
        \end{bmatrix}
        \begin{bmatrix}
            \dd W_1(s)\\
            \dd W_2(s)
        \end{bmatrix},
\end{align}
}
where 
\begin{align}
g_j(s; t) = \sigma_j a^*_{H_j} \left(f_+(s; t, H_j-1/2)+f_-(s; t, H_j-1/2)\right)     
\end{align}
for $s\in \mathbb{R}, t \geq 0, j = 1, 2$, $\sigma_1,\sigma_2>0$, and the constants $a^*_{H_j}$ are given by
\begin{align}
    \label{eq:C_wb}
    a^{*2}_H 
    & = \frac{2H(1-2H)\pi}{8\Gamma(2-2H)\cos(H\pi) \Gamma^2\left(H+1/2\right) \cos^2 \left(\frac{\pi (H-1/2)}{2}\right)}, 
\end{align}
to ensure the variances of the marginals $X_j^*(t)$ at time $t = 1$ are equal to $\sigma_j^2$.
\end{definition}
\noindent Similarly to the causal 2D fBm, the well-balanced 2D fBm is a zero-mean Gaussian process. 

\added[id=MB]{
The proposed constructions of 2D fBm are closely related to the general operator fractional Brownian motion (OFBM), a family of processes even more general than vector fractional Brownian motion. OFBM is a multivariate Gaussian process defined through a matrix-valued self-similarity exponent and integral kernels acting on multidimensional Brownian noise\cite{pipiras2017long,MAEJIMA1994139,Gustavo_2011,Gustavo_2018}. This connection provides helpful tools for the presented models and situates them within the broader theory of multivariate self-similar processes.}

\added[id=MB]{
Both of our considered models coincide with OFBM with the exponent matrix
$D~=~\mathrm{diag}\left(H_1 - \frac{1}{2},\, H_2 - \frac{1}{2}\right)$
and a kernel matrix determined by correlation parameter of the underlying noise $\rho$ and the model itself. Since in the causal 2D fBm the negative-time kernel is absent, the resulting process is inherently asymmetric in time. In contrast, the well-balanced 2D fBm corresponds to an OFBM with both positive and negative kernels present in a symmetric manner, resulting in a time-reversible process. This structural distinction is discussed in \autoref{sec:3} and is described through the asymmetry parameter~$\eta_{12}$.}

\added[id=MB]{
While OFBM theory typically imposes nontrivial admissibility conditions on the exponent and kernel matrices, our presented construction avoids these difficulties by employing correlated Brownian motions directly in the integral representation. As a consequence, the model is valid for all $H_1,H_2 \in (0,1)$ and $|\rho| \leq 1$ without additional constraints.}
\added[id=MB]{Thus, the introduced 2D fBm may be viewed as a subclass of OFBM, one that naturally captures both symmetric and asymmetric cross-dependence through the choice of causal or well-balanced form of integral kernels.}


\section{Covariance structure of 2D fBm}
\label{sec:3}
\noindent In this section, we discuss the covariance structure of both causal and well-balanced 2D fBms, along with that of their increments.
\subsection{Covariance structure of the process}
\noindent As mentioned, both causal and well-balanced 2D fBm are Gaussian processes; therefore, their covariance structure provides a complete characterization. In this section, we discuss this characterization in time domain.

\begin{theorem}
\label{thm:vfbm-cov}
Let $H_1, H_2 \in (0,1)$ and $|\rho|\leq1$. The covariance structure of 2D fBm $\mathbf{Z}(t)= [Z_1(t), Z_2(t)]'$, $ t\geq 0$ is as follows
\begin{align}
            \label{eq:cross_covariance}
            \gamma_{jk}(t, s) \equiv \langle Z_j(t)Z_k(s)\rangle = \frac{\sigma_j\sigma_k}{2}\left(w_{jk}(t) |t|^{H_j + H_k} + w_{jk}(-s) |s|^{H_j+H_k} - w_{jk}(t-s)|t-s|^{H_j+H_k} \right)
        \end{align}
        \noindent for $t, s \geq 0$, where $\sigma_j^2 = \langle Z_j^2(1)\rangle$ and 
        \begin{align}
        \label{eq:w_jk_causal}
        w_{jk}(u) = \begin{cases}
            \rho_{jk} - \eta_{jk} \sign(u), &\quad H_j+H_k \neq 1,\\
            \rho_{jk} - \eta_{jk} \sign(u)\log|u|, &\quad H_j+H_k =1.
        \end{cases}    
        \end{align}
        The, so called, cross-correlation parameters $\rho_{12}$ and $\rho_{21}$ are given by
        \begin{align}
            \rho_{12} &= \rho_{21}= \rho\frac{\sqrt{\Gamma(2H_1+1)\Gamma(2H_2+1)\sin(H_1\pi)\sin(H_2\pi)}}{\Gamma(H_1+H_2+1)\sin\left(\frac{H_1+H_2}{2}\pi\right) }\cos\left[\frac{(H_2-H_1)\pi}{2}\right],
        \label{eq:rho12}
        \end{align}
        and $\rho_{11}=\rho_{22}=1$
        while the asymmetry parameters $\eta_{jk}, j,k = 1, 2$, depend on the choice of the model, and in the causal 2D fBm (case $\mathbf{Z}(t)=\mathbf{X}(t)$) they are equal to 
        \begin{align}
            \eta_{12} &= -\eta_{21} =\rho\frac{\sqrt{\Gamma(2H_1+1)\Gamma(2H_2+1)\sin(H_1\pi)\sin(H_2\pi)}}{\Gamma(H_1+H_2+1)\cos\left(\frac{H_1+H_2}{2}\pi\right) }\sin\left[\frac{(H_2-H_1)\pi}{2}\right],
        \label{eq:eta12}
        \end{align}
        and $\eta_{11} = \eta_{22} = 0$, while for well-balanced 2D fBm (case $\mathbf{Z}(t)=\mathbf{X}(t)^*$) we have
        \begin{align}
            \label{eq:eta_wb}
            \eta_{11}=\eta_{22}=\eta_{12}=\eta_{21}=0.
        \end{align}
    \begin{proof}
    The proof of this theorem is presented in Appendix \autoref{appendix:proof_causal_cov} and Appendix \autoref{appendix:proof_well-balanced_cov}. 
    \end{proof}
\end{theorem}

\begin{remark}
    Since for well-balanced 2D fBm $\eta_{jk}=0$ for $j, k= 1, 2$, the covariance function simplifies remarkably. Similar to one-dimensional fBm, it is given by
    \begin{align}
        \gamma_{jk}(t, s) = \frac{\sigma_j\sigma_k \rho_{jk}}{2}\left(|t|^{H_j + H_k} + |s|^{H_j+H_k} - |t-s|^{H_j+H_k} \right).
    \end{align}
\end{remark}

\begin{remark}
\label{remark:3}
Let us note that for $H_1=H_2$ both causal and well-balanced 2D fBms define the same process.  
For $H_1=H_2=0.5$, both constructions lead to 2D Brownian motion with $\rho$ correlated coordinates\cite{balcerek2025two}.
\end{remark}

\begin{remark}
Let us observe that $\rho_{12}$ given in Eq. \eqref{eq:rho12} can be expressed as  
\begin{align}
        \rho_{12} = \langle Z_1(1)Z_2(1)\rangle/\sqrt{\langle Z_1^2(1) \rangle \langle Z_2^2(1) \rangle },
\end{align}
that is, it plays the role of the cross-correlation coefficient of $\mathbf{Z}(t)$ at time $t=1$.
\end{remark}

\noindent 
\autoref{thm:vfbm-cov} highlights the primary distinction between the two considered models. For the causal 2D fBm, we observe that when  $\eta_{12}\neq0$, the cross-covariance function $\gamma_{12}(s,t)$ is asymmetric, meaning $\gamma_{12}(s, t) \neq \gamma_{12}(t,s)$. This is in contrast to the marginal processes, for which $\gamma_{jj}(s, t)=\gamma_{jj}(t,s)$ for any suitable choice of distinct times $t$ and $s$. Conversely, for well-balanced 2D fBm, the cross-covariance function is symmetric (because $\eta_{12}=\eta_{21}=0$).
\\
\noindent The correspondence between the correlation coefficient of the underlying noise $\rho$ and the cross-correlation of the process, $\rho_{12}$, presented in Eq. (\ref{eq:rho12}), is shown in Figure \ref{fig:rhovsrhoeta12}(a) for different Hurst parameter pairs $(H_1, H_2)$ of the coordinates $X_1(t)$ and $X_2(t)$. Let us highlight the fact that, when the difference between $H_1$ and $H_2$ is large (e.g., $H_1=0.2, H_2=0.7$, yellow line in the figure) the resulting cross-correlation range is much smaller than the correlation of the underlying noise. In the extreme case of $|\rho| = 1$, the parameter $|\rho_{12}|\approx 0.5$. When the difference between $H_1$ and $H_2$ is small (e.g., $|H_1-H_2|=0.25 \pm0.05$, orange and violet lines in the figure), the resulting $\rho_{12}$ is similar to the correlation $\rho$.

\noindent The correspondence between $\rho$ and the asymmetry parameter $\eta_{12}$, as given in Eq. (\ref{eq:eta12}), is presented in Figure \ref{fig:rhovsrhoeta12}(b). We observe that $\eta_{12}$ increases with the difference between $H_1$ and $H_2$, while $\rho_{12}$ decreases.
\begin{figure}[ht!]
\centering
     \includegraphics[width=\textwidth]{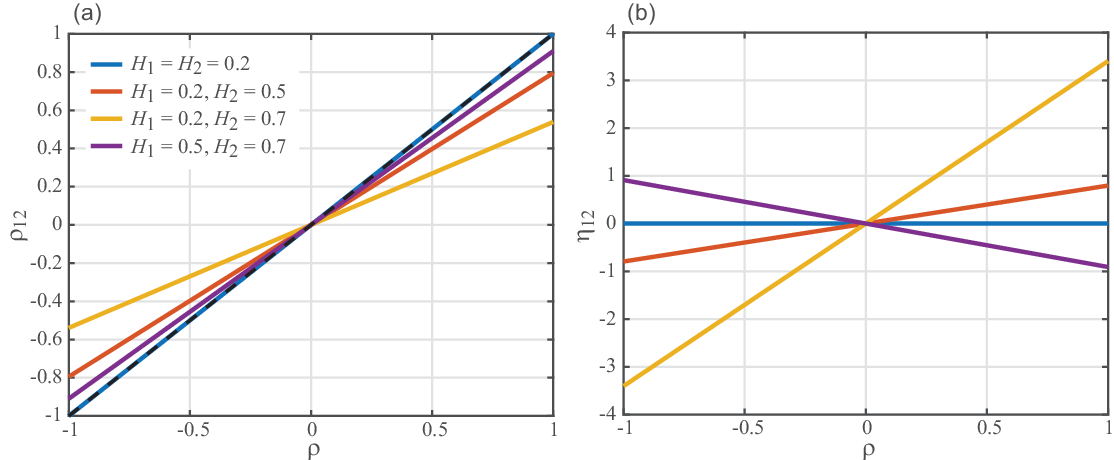}
     \caption{The dependence between correlation $\rho$ of the underlying noise and (a) cross-correlation coefficient $\rho_{12}$ and (b) asymmetry parameter $\eta_{12}$ of the process $\mathbf{X}(t)$ and $\mathbf{X}^*(t)$. Different solid lines correspond to different Hurst exponents of the coordinates, while the black dashed line corresponds to the identity cases $\rho_{12}=\rho$ or $\eta_{12}=\rho$, shown as a guide to the eye.}
 \label{fig:rhovsrhoeta12}    
\end{figure}

\begin{remark} Under the presented construction, the introduced processes $\mathbf{X}(t)$ and $\mathbf{X^*}(t)$ are well-defined for all sets of parameters $H_1, H_2 \in (0,1)$ and $|\rho|\leq1$. Such a natural construction ensures that the condition given in Remark 8 in \cite{amblard2010basic} is true. Moreover, we see that $\rho_{12}^2 \leq \rho^2$.
\end{remark}

\subsection{Covariance structure of increments}
\noindent For any process $\mathbf{Z}(t)$, the increment process over a step $\delta > 0$ is defined as
\begin{align}
    \Delta^\delta \mathbf{Z}(t) \equiv \mathbf{Z}(t+\delta) - \mathbf{Z}(t).
\end{align}
In the following theorem, we present the covariance structure for the increment process of the causal and well-balanced 2D fBm. 
\begin{theorem} 
\label{thm:causal-inc-cov}
Let $H_1, H_2 \in (0,1)$ and $|\rho|\leq1$. The covariance structure of the increment process $\Delta^\delta \mathbf{Z}(t)$ of causal or well-balanced 2D fBm is as follows
\begin{align}
    \gamma_{jk}^\Delta(t, s)
    & = \langle \Delta^\delta Z_j(t) \Delta^\delta Z_k(s) \rangle = \frac{\sigma_j \sigma_k \rho_{jk}}{2}\left[w_{jk}(t-s+\delta) |t-s+\delta|^{H_j+H_k} \right. \nonumber \\
     & \qquad + \left. w_{jk}(t-s-\delta) |t-s-\delta|^{H_j+H_k} - 2w_{jk}(t-s) |t-s|^{H_j+H_k}  \right],
\end{align}
where the function $w_{jk}$ is given in Eq.\eqref{eq:w_jk_causal} and depends on whether we consider causal or well-balanced 2D fBm. The coefficients $\rho_{jk}$ and $\eta_{jk}$ are defined in \autoref{thm:vfbm-cov}.

\end{theorem}

\begin{proof}
\noindent The proof of this theorem is a natural consequence of the covariance structure of the introduced processes $\mathbf{X}(t)$ and $\mathbf{X}^*(t)$. For the increment process of the causal 2D fBm $\mathbf{X}(t)$ we have
    \begin{align}
    \langle \Delta^\delta X_j(t) \Delta^\delta X_k(s) \rangle &= \langle (X_j(t+\delta)-X_j(t)) (X_k(s+\delta) - X_k(s))\rangle \nonumber \\ 
    & = \gamma_{jk}(t+\delta, s+\delta) - \gamma_{jk}(t, s+\delta) - \gamma_{jk}(t+\delta, s) + \gamma_{jk}(t, s),\quad j,k=1,2.
    \end{align}
 Substituting the covariance function $\gamma_{jk}$ using the results of \autoref{thm:vfbm-cov} yields the desired covariance structure of the increment process. The proof for well-balanced 2D fBm is analogous.   
\end{proof}
\noindent Let us note that for both considered processes, the covariance structure depends only on $h= t-s$ and the length of the interval $\delta$ on which the increments are taken. For $j=k$, we retain the autocovariance function of fractional Brownian motion corresponding to that coordinate\cite{beran2016long}, i.e., 
\begin{align}
    \langle \Delta^\delta X_j(t) \Delta^\delta X_j(s) \rangle = \langle \Delta^\delta X^*_j(t) \Delta^\delta X^*_j(s) \rangle &= \frac{\sigma_j^2}{2} \left( |t-s+\delta|^{2H_j} + |t-s-\delta|^{2H_j} - 2|t-s|^{2H_j}\right).
\end{align}




\section{Spectral Content}
\label{sec:4}
\noindent The PSD matrix plays a central role in the analysis of multidimensional stochastic processes\cite{white2002cross}.
This quantity generalizes the concept of a PSD to vector-valued processes. 
It is particularly valuable in multiple fields, including econometrics \cite{granger1983applications}, structural engineering \cite{de2014fatigue,kerschen2006past}, neuroscience \cite{wang2015power}, and signal processing \cite{amiri2013nature}, where complex systems exhibit interactions across multiple channels or variables. Since PSD is more frequently employed in the context of stationary processes, we begin this section by analyzing the increments of 2D fBm.

\subsection{Spectral content of the  increments}
\noindent In the multivariate setting, for any second-order stationary process $\Delta \mathbf{Z}(t)$, $t\geq 0,$ we can consider the power spectral density matrix $S_{\Delta \mathbf{Z}}(f)$ of a single realization
\begin{equation}
    S_{\Delta \mathbf{Z}}(f) =  \lim_{T\to\infty} \frac{1}{T}\int_{0}^T \int_0^T e^{i (t-s) f} \Delta \mathbf{Z}(t)\Delta \mathbf{Z}(s)' \,\dd t \dd s, \quad f \in \mathbb{R},    
\label{Eq:PSDst}
\end{equation}
with ensemble average
\begin{equation}
    \langle S_{\Delta \mathbf{Z}}(f)\rangle=
    \lim_{T\to\infty}\frac{1}{T}\int_{0}^T\int_0^T e^{i (t-s) f} \gamma_{\Delta \mathbf{Z}}(t-s) \,\dd t \dd s, \quad f \in \mathbb{R},
\label{Eq:WK}
\end{equation}
where $\gamma_{\Delta \mathbf{Z}}(n) = \langle \Delta \mathbf{Z}(n) \Delta \mathbf{Z}(0)'\rangle$ for the zero-mean process $\Delta\mathbf{Z}(t)$. Eq. (\ref{Eq:WK}) simplifies to
\begin{align}
   \langle S_{\Delta \mathbf{Z}}(f)\rangle = 
    \int_{-\infty}^\infty e^{i t f } \gamma_{\Delta\mathbf{Z}}(t) \,\dd t, \quad f \in \mathbb{R},
\end{align}
which is known as Wiener-Khinchin theorem\cite{kubo2012statistical}.
Let us note, that for a $d$-dimensional process, the function $S_{\Delta\mathbf{Z}}$ is a $d\times d$ matrix with the diagonal elements corresponding to marginals and off-diagonal elements describing the, so called, cross power spectral density \cite{white2002cross}.
To streamline the notation, we will use now $S^\Delta$ whenever PSD relates to the increments process, instead of using $S_{\Delta \mathbf{Z}}$.

\noindent One can also consider the components of a power spectral density matrix
\begin{align}
   \langle S^\Delta_{jk}(f)\rangle = 
    \int_{-\infty}^\infty e^{i t f } \gamma^\Delta_{jk}(t) \,\dd t, \quad f \in \mathbb{R},~~j, k=1,2.
\end{align}
It is worth noting that unlike the one-dimensional version of spectral density, the cross components $S^\Delta_{jk}$ for $j\neq k$ might not be real-valued. In general, it is true that $S^\Delta_{jk}(f) = \overline{S^\Delta_{kj}(f)}$ for all $f$'s, where $\overline{z}$ denotes the complex conjugate of $z$.

\begin{theorem}
\label{thm:spectral_vfbm_inc}
Let $H_1, H_2 \in (0,1)$, $H = \mathrm{diag}(H_1, H_2)$ be a diagonal matrix with elements $H_1, H_2$, and $|\rho|\leq1$. Additionally, let matrices $C_\textrm{c} = [c^\textrm{c}_{jk}]_{j, k=1,2}$ and $C_\textrm{wb}= [c^\textrm{wb}_{jk}]_{j, k=1,2}$ that correspond to causal and well-balanced 2D fBm, respectively, have the following elements
\begin{align}
    c^\textrm{c}_{jj}&=\frac{1}{2\pi} \sigma^2_j \Gamma^2(H_j+0.5)a^2_{H_j},\\
    c^\textrm{c}_{jk} &= \frac{1}{2\pi} \rho\sigma_j\sigma_k \Gamma(H_j+0.5)\Gamma(H_k+0.5)a_{H_j}a_{H_k}e^{-i\frac{\pi}{2}(H_j-H_k)},  \quad j\neq k,\\
    c^\textrm{wb}_{jj}&=\frac{2}{\pi} \sigma^2_j \cos^2\left(\frac{(H_j-0.5)\pi}{2}\right)\Gamma^2(H_j+0.5)a^{*2}_{H_j},\\
    c^\textrm{wb}_{jk} &= \frac{1}{2\pi} \rho\sigma_j\sigma_k\cos\left(\frac{(H_j-0.5)\pi}{2}\right)\cos\left(\frac{(H_k-0.5)\pi}{2}\right) \Gamma(H_j+0.5)\Gamma(H_k+0.5)a^*_{H_j}a^*_{H_k}, \quad j\neq k,
\end{align}
for $j, k =1,2$.
Constants $a_{H_j}$ and $a^*_{H_j}$ are given in Eqs. \eqref{eq:C_causal}) and  (\ref{eq:C_wb}), respectively.
Then, for the increments of 2D fBm 
the power spectral density matrix $S^\Delta(f)$ is given by 
\begin{align}
    \langle S^\Delta(f)\rangle 
    & = \left|1-e^{-if}\right|^2 \sum_{n=-\infty}^\infty \left[(f+2\pi n)_+^{-D} \tilde{C}(f + 2\pi n)_+^{-D} + (f+2\pi n)_-^{-D} \overline{\tilde{C}} (f + 2\pi n)_-^{-D} \right]/(f+2\pi n)^2,
\end{align}
where $(x)_+ \equiv \max\{x, 0\}, (x)_- \equiv \max\{-x, 0\}, D = \mathrm{diag}(H_1, H_2) - 0.5 I_{2}$, \added[id=MB]{$I_2$ is} \added[id=DK]{a $2\times2$ identity matrix,} and $\overline{C}$ denotes the element-wise complex conjugate of matrix $C$. Matrix $\tilde C$ is equal to $C_\textrm{c}$ or $C_\textrm{wb}$ depending if we consider causal or well-balanced 2D fBm. The components $\langle S^\Delta_{jk} \rangle$ can be thus expressed as
\begin{align}
    \langle S^\Delta_{jk}(f)\rangle = |1 - e^{-if}|^2 \sum_{n=-\infty}^\infty \left[ (f+2\pi n)_+^{1-H_j-H_k} \tilde{c}_{jk} + (f+2\pi n)_-^{1-H_j-H_k} \overline{\tilde{c}}_{jk}\right]/(f+2\pi n)^2,
\end{align}
where $\tilde{c}_{jk}$ are elements of the appropriate matrix, $C_\textrm{c}$ for the causal case or $C_\textrm{wb}$ for the well-balanced one. Moreover, for $f \to 0$ we have 
\begin{align}
    \langle S^\Delta(f)\rangle \sim f^{-D} \tilde{C} f^{-D}.
\end{align}
The element-wise asymptotic is as follows
\begin{align}
    \langle S^\Delta_{jk}(f)\rangle \sim \tilde{c}_{jk} f^{-(d_j + d_k)} = \tilde{c}_{jk} f^{-(H_j + H_k -1)}.
\end{align}
\begin{proof}
The proof of this theorem is presented in Appendix \autoref{appendix:proof_psd_inc}.
\end{proof}
\end{theorem}

\begin{remark}
    Matrices $C_\textrm{c}$ and $C_\textrm{wb}$ given in \autoref{thm:spectral_vfbm_inc} have real elements on the diagonal, thus, $\langle S^\Delta_{jj}\rangle, j=1,2$, is also a real (and non-negative) function.
\end{remark}
\begin{remark}
    The direct derivation of matrix $C$ given in \autoref{thm:spectral_vfbm_inc} be found in the proof of \autoref{thm:vfbm-cov} given in Appendix \autoref{appendix:proof_causal_cov} and Appendix \autoref{appendix:proof_well-balanced_cov}.
\end{remark}




\subsection{Spectral content of the process}
\noindent An extension of PSD calculated on the stationary process (here, for the increments process of 2D fBm), one considers PSD calculated based on the information from the trajectory. In particular, given a nonstationary 1D process $Z(t)$, $0\leq t \leq T$, measured over a time $T$, the PSD is typically defined as \cite{norton2003fundamentals,niemann2013fluctuations,krapf2018power,krapf2019spectral} 
\begin{align}\label{eq:psd}
    S_{Z}(f, T) \equiv \frac{1}{T} \left|\int_0^T e^{i f t} Z(t)\, \dd t\right|^2.
\end{align}
Note that in contrast to Eq.~\eqref{Eq:PSDst}, the PSD depends on measurement time $T$. In $d$ dimensions, the natural extension of the PSD defined in Eq. (\ref{eq:psd}) for a real process $\mathbf{Z}(t)$, $ 0\leq t \leq T $,  can be written in an equivalent form,
\begin{align}
    \label{eq:cpsd_general}
    S_\mathbf{Z}(f, T) = \frac{1}{T} \int_0^T e^{i f t} \mathbf{Z}(t)\, \dd t \left(\int_0^T e^{i f s} \mathbf{Z}(s)\, \dd s\right)^*,
\end{align}
where $(\mathbf{x})^*$ is the Hermitian transpose (also known as the conjugate transpose) of $\mathbf{x}$, i.e., $(\mathbf{x})^* = \overline{\mathbf{x}}'$. Again, the function $S_\mathbf{Z}$ is a $d\times d$ matrix with the diagonal elements corresponding to marginals, and off-diagonal elements corresponding to the cross power spectral density.
\noindent Eq. \eqref{eq:cpsd_general} can be written using a double integral
\begin{align}
    S_{\mathbf{Z}}(f, T) = \frac{1}{T} \int_0^T \int_0^T e^{i f (t-s)} \mathbf{Z}(t) \mathbf{Z}'(s) \dd s\, \dd t.
\end{align}
Taking the expected value, we can also calculate the ensemble-averaged PSD
\begin{align}
    \langle S_{\mathbf{Z}}(f, T)\rangle \equiv \frac{1}{T} \int_0^T \int_0^T e^{i f (t-s)} \langle \mathbf{Z}(t) \mathbf{Z}'(s)\rangle \dd s\, \dd t, 
\end{align}
alternatively, written component-wise it is
\begin{align}
    \langle S_{\mathbf{Z}, jk}(f, T)\rangle \equiv \frac{1}{T} \int_0^T \int_0^T e^{i f (t-s)} \langle Z_j(t) Z_k(s)\rangle \dd s\, \dd t, ~~j, k = 1, 2.
\end{align}
Let us note that for $j=k$ we obtain the classical PSD given in Eq. (\ref{eq:psd}) of the 1D process corresponding to the marginals.
Here, we further use the simplified notation $S_{jk}$ instead of $S_{\mathbf{Z}, jk}$ to streamline the notation whenever it is clear which process is referred to.
In both cases (causal and well-balanced),  ensemble-averaged PSD for 2D fBm can be expressed as
{\small \begin{align}
    \langle S_{jk}(\tilde{\omega}, T)\rangle = T^{H_j + H_k + 1} \frac{\sigma_j \sigma_k}{2} \int_0^1 \int_0^1 e^{i \tilde{\omega} (t-s)} \left(w_{jk}(x) x^{H_j+H_k} + w_{jk}(-y) y^{H_j+H_k} - w_{jk}(x-y) |x-y|^{H_j+H_k}  \right)\dd x\, \dd y, 
\end{align}}
where $\tilde{\omega} = f T$. 
The function $w_{jk}$ is defined in Eq. (\ref{eq:w_jk_causal}) with parameters $\rho_{jk}$ and asymmetry parameters $\eta_{jk}$ defined in \autoref{thm:vfbm-cov}. In the following theorem, we present the behavior of such a function.

\begin{theorem}
\label{thm:spectral_vfbm}
For the 2D fBm with $H_1+H_2\neq 1$, the ensemble-averaged PSD for coordinates $j,k= 1, 2$ has the form 
\begin{align}
\langle S_{jk}(\tilde{\omega}, T)\rangle = T^{H_j + H_k + 1} \sigma_j \sigma_k &\left\{
\rho_{jk} \left[\frac{1-\cos \tilde{\omega}}{\tilde{\omega}}\mathcal{S}_{jk} - \left(1 - \frac{\sin \tilde{\omega}}{\tilde{\omega}}\right) \mathcal{C}_{jk} + \frac{\dd}{\dd\tilde{\omega}} \mathcal{S}_{jk} \right] \right. \nonumber \\
+ & \eta_{jk} \left. i \left[\frac{1-\cos \tilde{\omega}}{\tilde{\omega}}\mathcal{C}_{jk} - \left( 1 + \frac{\sin \tilde{\omega}}{\tilde{\omega}}\right) \mathcal{S}_{jk} - \frac{\dd}{\dd\tilde{\omega}} \mathcal{C}_{jk}\right] \right\},
\end{align}
where 
\begin{align}\label{eq:th3_1}
    \mathcal{C}_{jk} \equiv \mathcal{C}_{jk}(\tilde{\omega})&= \int_0^1 \cos(\tilde{\omega} x) x^{H_j+H_k} \,\dd x,\end{align}
 \begin{align}  \label{eq:th3_2} \mathcal{S}_{jk} \equiv \mathcal{S}_{jk}(\tilde{\omega})&= \int_0^1 \sin(\tilde{\omega} x) x^{H_j+H_k} \,\dd x.
\end{align}
\noindent Parameters $\rho_{jk}$ and $\eta_{jk}$ are given in \autoref{thm:vfbm-cov}.

\begin{proof}
    The proof of this theorem is included in Appendix \autoref{appendix:proof_psd}.    
\end{proof}
\end{theorem}

\begin{remark} We can also use alternative forms for functions $\mathcal{C}$ and $\mathcal{S}$ defined in Eqs. \eqref{eq:th3_1} and \eqref{eq:th3_2}, respectively, to obtain the expressions without derivatives. Alternatively, using integration by parts and the identities 
\begin{align}
    \frac{\dd}{\dd\tilde{\omega}} \mathcal{S}_{jk} &= \frac{\sin \tilde{\omega}}{\tilde{\omega}} - \frac{H_j + H_k + 1}{\tilde{\omega}} \mathcal{S}_{jk},\\
    \frac{\dd}{\dd\tilde{\omega}} \mathcal{C}_{jk} &= \frac{\cos \tilde{\omega}}{\tilde{\omega}} - \frac{H_j + H_k + 1}{\tilde{\omega}} \mathcal{C}_{jk},
\end{align} 
we obtain the equivalent expression
\begin{align}
\langle {S}_{jk}(\tilde{\omega}, T)\rangle = T^{H_j + H_k + 1} \sigma_j \sigma_k &\left\{
\rho_{jk} \left[\frac{1-\cos \tilde{\omega} - H_j - H_k - 1}{\tilde{\omega}}\mathcal{S}_{jk} - \left(1 - \frac{\sin \tilde{\omega}}{\tilde{\omega}}\right) \mathcal{C}_{jk} + \frac{\sin \tilde{\omega}}{\tilde{\omega}} \right] \right. \nonumber \\
+ & \eta_{jk} \left. i \left[\frac{1-\cos \tilde{\omega} + H_j + H_k + 1}{\tilde{\omega}}\mathcal{C}_{jk} - \left( 1 + \frac{\sin \tilde{\omega}}{\tilde{\omega}}\right) \mathcal{S}_{jk} - \frac{\cos \tilde{\omega}}{\tilde{\omega}}\right] \right\}.
\end{align}
\noindent For marginals (i.e., when $j = k$), the expression reduces to the known formula for PSD of 1-dimensional fBm with Hurst parameter $H = H_j = H_k$ \cite{krapf2019spectral}, since $\eta_{jj} = 0$. In contrast, for $j \neq k$, the cross power spectral density acquires an imaginary part in the causal case, reflecting time-asymmetric dependence between components. This imaginary component vanishes in the well-balanced case, where $\eta_{jk} \equiv 0$ (cf. Eq. \eqref{eq:eta_wb}).
\end{remark}

\begin{theorem}[Asymptotic behavior of the ensemble-averaged PSD]
\label{thm:psd_asymptotics}
The ensemble-averaged PSD $\langle S_{jk}(\tilde{\omega}, T)\rangle$ of the 2D fBm admits the following asymptotic regimes as $\tilde{\omega} \to \infty$ if $H_j+H_k\neq 1$ or for well-balanced 2D fBm
\begin{align}
    \Re\langle S_{jk}(\tilde{\omega}, T) &\rangle \sim \quad T^{H_j + H_k + 1} \sigma_j \sigma_k \rho_{jk} \left\{\left[\frac{1}{\tilde{\omega}^2} - \frac{(H_j+H_k)\sin\tilde{\omega}}{\tilde{\omega}^3} + O\left(\frac{1}{\tilde{\omega}^4}\right) \right]  \right. \nonumber \\
    & \left. + a_{jk}\left[\frac{1}{\tilde{\omega}^{H_j + H_k + 1}} - \frac{(H_j+H_k+\cos\tilde{\omega}) \cot\left(\frac{\pi}{2}(H_j+H_k) \right)+ \sin \tilde{\omega}}{\tilde{\omega}^{H_j + H_k + 2}}\right] \right\},\\
    \Im\langle S_{jk}(\tilde{\omega}, T) \rangle &\sim \quad T^{H_j + H_k + 1} \sigma_j \sigma_k \eta_{jk} \left\{\left[\frac{2}{\tilde{\omega}^2} - \frac{3-\cos\tilde{\omega}}{\tilde{\omega}^3} + O\left(\frac{1}{\tilde{\omega}^4}\right) \right]  \right. \nonumber \\
    & \left. -a_{jk} \left[\frac{\cot\left(\frac{\pi}{2}(H_j+H_k) \right)}{\tilde{\omega}^{H_j + H_k + 1}} + \frac{2+H_j +H_k - \cos \tilde{\omega} + \sin\tilde{\omega} \cot\left(\frac{\pi}{2}(H_j+H_k) \right)}{\tilde{\omega}^{H_j+H_k +2}}\right] \right\},
\end{align}
where $\Re$ and $\Im$ denote the real and imaginary parts, respectively.
\noindent Parameters $\rho_{jk}$ and $\eta_{jk}$ are given in \autoref{thm:vfbm-cov}. In the formulas above, the parameters $a_{jk} = \Gamma(H_j + H_k) \sin\left(\frac{\pi}{2}(H_j+H_k)\right)$.
\begin{proof}
The proof of this theorem is based on the asymptotic expansion of the incomplete gamma function and is presented in details in Appendix \autoref{appendix:proof_psd_asymptotics}.
\end{proof}
\end{theorem}
\noindent  The behavior of ensemble-averaged PSD od 2D fBm depends explicitly on the Hurst parameters $H_j$, $H_k$, as well as the parameters $\rho_{jk}$ and $\eta_{jk}$. In the following remark we present the details.

\begin{remark} The asymptotic behavior of the ensemble-averaged PSD $\langle S_{jk}(\tilde{\omega}, T)\rangle$ for 2D fBm depends on whether the sum of Hurst parameters $H_j+H_k$ is greater than or smaller than 1. In particular, when $H_j+H_k>1$, we observe dependence on measurement time $T$. More precisely, we have
\begin{align}
\Re \langle S_{jk}(f, T)\rangle &\sim \sigma_j\sigma_k\rho_{jk} \left[\frac{1}{f^2}T^{H_j+H_k-1} + \frac{a_{jk}}{f^{H_j+H_k+1}} + o(1)\right],\\
\Im \langle S_{jk}(f, T)\rangle &\sim  \sigma_j\sigma_k\eta_{jk} \left[\frac{2}{f^2}T^{H_j+H_k-1} - \frac{a_{jk}\cot\left(\frac{\pi}{2}(H_j+H_k) \right)}{f^{H_j+H_k+1}} + o(1)\right],
\end{align}
Conversely, for $H_j+H_k<1$, the decay is faster, and the cross-component interactions diminish more rapidly with frequency
\begin{align}
\Re \langle S_{jk}(f, T)\rangle &\sim \sigma_j\sigma_k \rho_{jk}\left[\frac{a_{jk}}{f^{H_j+H_k+1}} + \frac{1}{f^2 T^{1-H_j-H_k}} + o\left(T^{H_j+H_k-1}\right)\right],\\
\Im \langle S_{jk}(f, T)\rangle &\sim \sigma_j\sigma_k \eta_{jk}\left[\frac{-a_{jk}\cot\left(\frac{\pi}{2}(H_j+H_k) \right)}{f^{H_j+H_k+1}} + \frac{2}{f^2 T^{1-H_j-H_k}} + o\left(T^{H_j+H_k-1}\right)\right]
\end{align}
This dichotomy highlights the crucial role of $H_j+H_k$ on the spectral structure of the process.
\end{remark}

\section{Numerical simulations}\label{sec:5}
\noindent In this section, we present the comparison of the analytical results obtained in the previous parts of this article with numerical simulations. For the numerical simulations of 2D fBm we apply the Wood and Chan's algorithm\cite{chan1999simulation}, which is based on embedding the covariance matrix of 2D fBm increments into a circulant matrix. All of the presented numerical results are based on the ensemble of $N=5,000$ trajectories of length $T=2^{16}=65,536$ with a time step $1$.

\subsection{Trajectories}
\noindent To visualize 2D fBm and evaluate our analytical results, we performed extensive numerical simulations. Figures \ref{fig:trajs-causal2} and \ref{fig:trajs-causal3} show sample trajectories of the 2D fBm. In both figures, upper row corresponds to the case of independent coordinates, i.e., $\rho = \rho_{12} = 0$, while the bottom row corresponds to the case $\rho_{12}=0.5$. In Figure \ref{fig:trajs-causal2} we present the case with $H_1=H_2$, that is, a case when causal and well-balanced processes coincide. 
In Figure \ref{fig:trajs-causal3} we focus on the case with $H_1 \neq H_2$ and the trajectories are realizations of the causal 2D fBm. We present more trajectories to highlight the difference in the behavior of the process in the cardinal directions. For example, on panels (b) and (d), the process is characterized by subdiffusion on the horizontal axis (1st coordinate) while on the vertical it is a superdiffusion (2nd coordinate).
\begin{figure}[ht!]
    \centering
    \includegraphics[width=\textwidth]{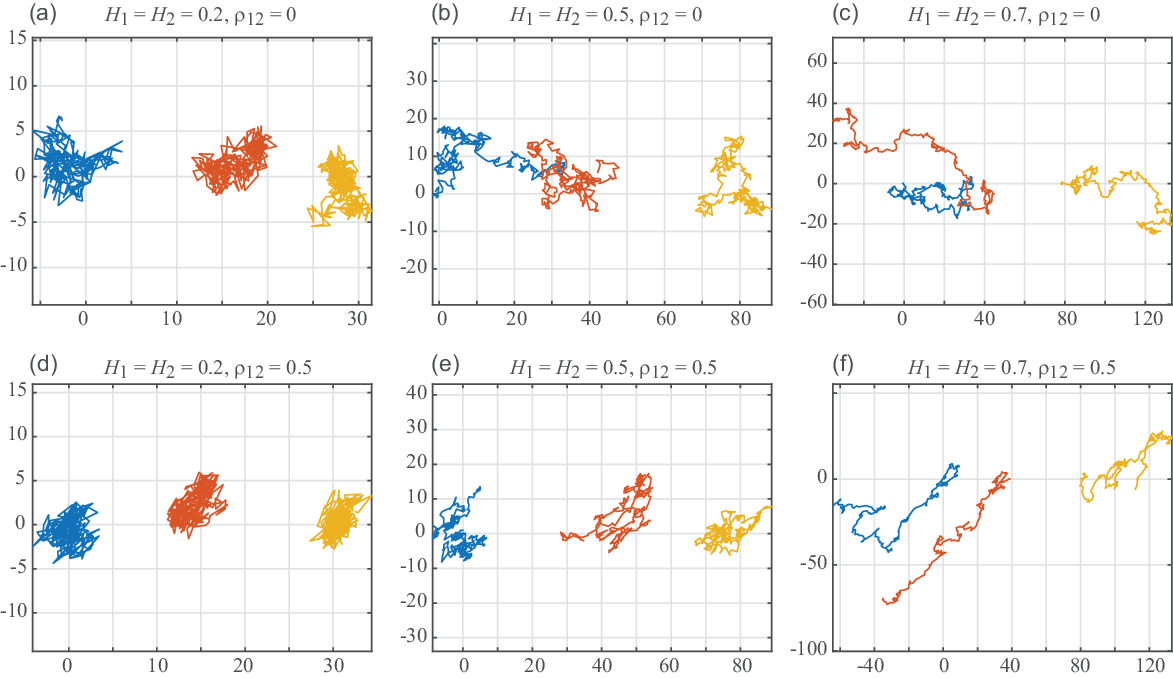}

    \caption{Sample trajectories of causal 2D fBm with $H_1=H_2=H$. Each panel corresponds to a given set of parameters $(H, \rho_{12})$, where the top row (a-c) corresponds to uncorrelated components and the bottom row (d-f) to cross-correlation $\rho_{12}=0.5$. Pairs of panels in each column have the same $H$. The trajectories are shifted on the horizontal axis for clarity.}
    \label{fig:trajs-causal2}
\end{figure}

\newpage

\begin{figure}[ht!]
    \centering
    \includegraphics[width=\textwidth]{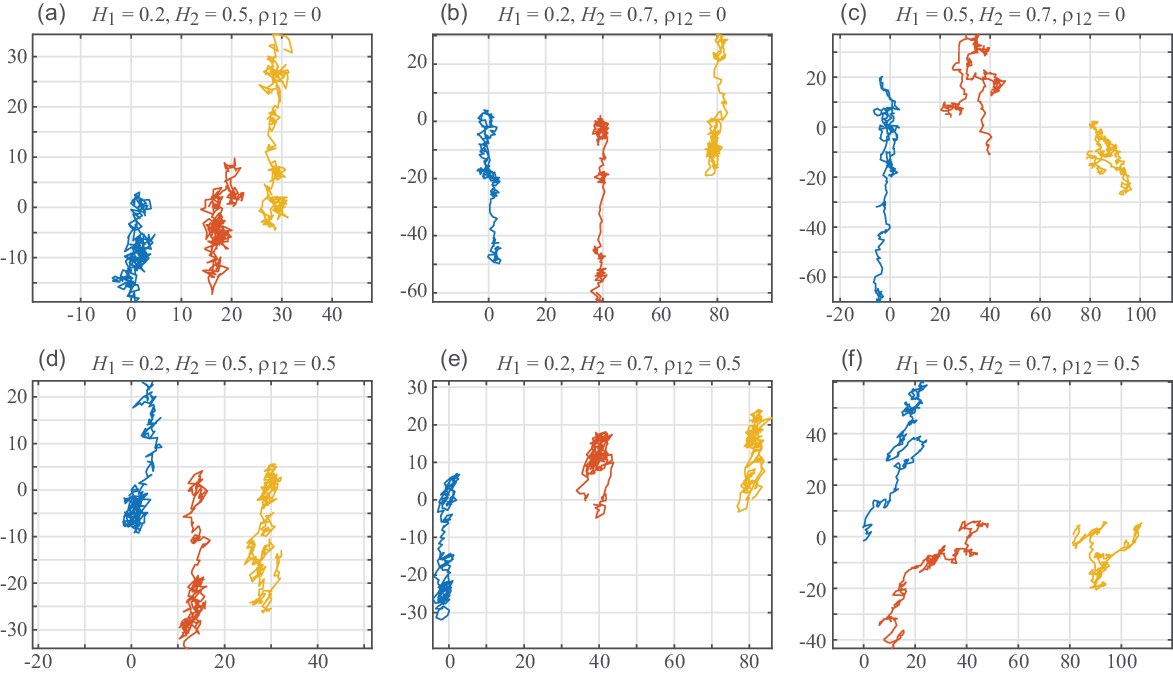}
  
    \caption{Sample trajectories of causal 2D fBm with different Hurst exponents, $H_1 \neq H_2$. Each panel corresponds to a given set of parameters $(H_1, H_2, \rho_{12})$, where the top row (a-c) corresponds to uncorrelated components and the bottom row (d-f) to cross-correlation $\rho_{12}=0.5$. The trajectories are shifted on the horizontal axis for clarity.}
        \label{fig:trajs-causal3}
\end{figure}
\subsection{Autocovariance}
\noindent The different cases of cross-covariance functions $\gamma_{12}$ given in \autoref{thm:vfbm-cov} are presented in Figures \ref{fig:ccov_HH} and \ref{fig:ccov_H1H2}. In both of these figures, the cross-covariance function for the causal version of the process is plotted using a solid line, while for well-balanced case, a dashed line is used. 
\autoref{fig:ccov_HH} showcases the situation when $H_1=H_2$. However, as discussed in Remark \autoref{remark:3}, in such a case both causal and well-balanced cases are the same -- that is why the continuous and dashed lines overlap each other. We see that the cross-covariance function behaves similarly to that we expect for the fractional Brownian motion with the corresponding $H$. The only difference between cross-covariances of the causal and well-balanced processes is the rescaling of such function by a factor of $\rho_{12}$. Thus, each case (subdiffusive $H_1=H_2=0.2$ -- blue line, left panel; diffusive $H_1=H_2=0.5$ -- orange line, middle panel; superdiffusive $H_1=H_2=0.7$ -- yellow line, right panel) retains the behavior.
\\
\noindent A more interesting situation is presented in Figure \ref{fig:ccov_H1H2}, where the corresponding Hurst parameters $H_1$ and $H_2$ are not equal. Here, the dashed lines corresponding to the well-balanced case are different from the ones for causal 2D fBm. Particularly, it seems that the `strength' of the cross-covariance is usually smaller (in the absolute sense) in such cases. The most interesting is the case of $H_1=0.2, H_2=0.7, \rho_{12}=0.5$ (orange line, middle panel) -- for which we change the type of memory for lags $h>0$. For the causal case, the cross-covariance function was positive, while the addition of the well-balanced element makes it negative.
Another interesting effect can be observed for the causal version for  $H_1=0.2, H_2=0.5$ and $H_1=0.5, H_2=0.7$ and $\rho_{12}=0.5$ in both cases. The cross-covariance function disappears for $h>0$ and $h<0$, respectively. It is due to the fact that via the choice of parameters we obtained $\rho_{12} - \eta_{12}\sign{h}=0$ thus cancelling every element in the cross-covaraince function for the increments for lags $|h|\geq1$ (cf. the relation between $\rho_{12}$ and $\eta_{12}$ in Eq.~\eqref{eq:eta_vs_rho} of the Appendix).

\begin{figure}[ht!]
     \centering
         \centering
               \includegraphics[width=\textwidth]{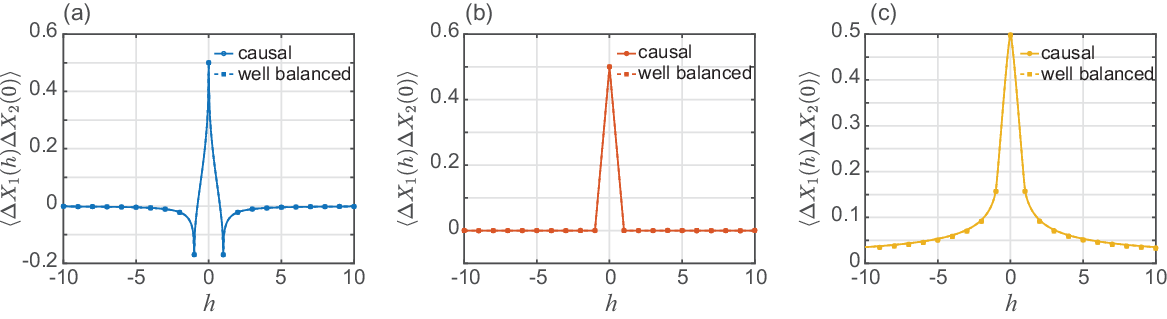}
               \caption{Cross-covariance function depending on $H_1, H_2$ and $\rho_{12}=0.5$ for $H_1=H_2$. (a) $H_1 = H_2 = 0.2$, (b) $H_1 =H_2=0.5$, and (c) $H_1 = H_2 = 0.7$.  
               The solid lines represent the causal version of the model and dashed lines the well-balanced case. Markers correspond to the estimated values of the cross-covariance function; circles correspond to the causal case, squares to the well-balanced case. Here, casual and well-balanced cases overlap, as for $H_1=H_2$ the model is time-reversible regardless of the definition (cf. Eq. (\ref{eq:eta12})). \added[id=MB]{Increments are taken over a unit time interval, i.e., $\Delta \equiv \Delta^{\delta=1}$.}} 
               \label{fig:ccov_HH}
\end{figure}

\begin{figure}[ht!]
     \centering
     \includegraphics[width=\textwidth]{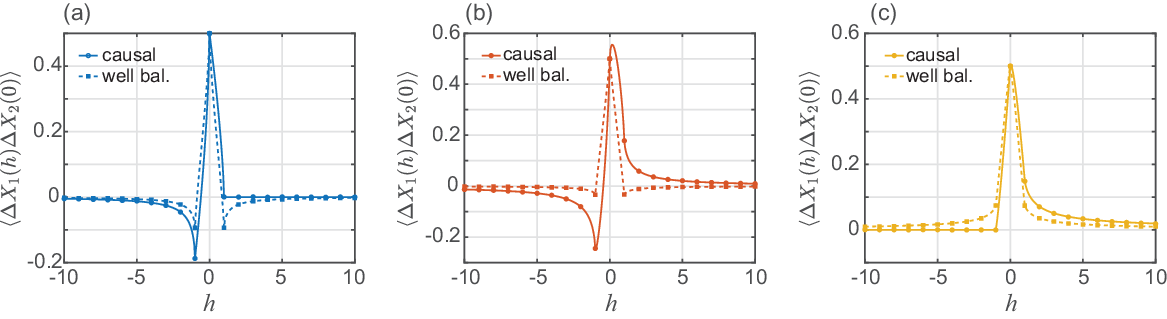}
     \caption{Cross-covariance function depending on $H_1, H_2$ and $\rho_{12}$ for $H_1\neq H_2$. (a) $H_1 = 0.2, H_2=0.5$, (b) $H_1 = 0.2, H_2=0.7$, and (c) $H_1 = 0.5, H_2 = 0.7$. Solid lines represent the causal version of the model, the dashed lines the well-balanced case. Markers correspond to the estimated values of the cross-covariance function; circles correspond to the causal case, squares to the well-balanced case.}
     \label{fig:ccov_H1H2}
\end{figure}

\subsection{Spectral content}
\noindent In Figure \ref{fig:psd1} we present a comparison between estimated ensemble-averaged (cross-)PSD to their theoretical asymptotics for the 2D fBm. As in previous figures, solid lines represent estimates from simulations, while dashed lines indicate the theoretical asymptotics. Different colored solid lines correspond to different sets of parameters $(H_1, H_2)$, whereas $\rho_{12}=0.5$ for all the cases.
Panels (a) and (b) present the real part, i.e., $\Re\langle S_{\mathbf{X}, 12}(f, T)\rangle$, with panel (a) corresponding to the case with $H_1=H_2$ (i.e., when the causal and well-balanced 2D fBms coincide), while panel (b) represents the causal 2D fBm with $H_1\neq H_2$. 
Panel (c) showcases the imaginary part, i.e., $\Im\langle S_{\mathbf{X}, 12}(f, T)\rangle$ for the sets $H_1\neq H_2$ corresponding to panel (b) (as for the $H_1=H_2$ case, the cross-PSD is 0).
\\
\noindent In all panels, we see that the asymptotic behavior of the ensemble-averaged cross-PSD depends on the $H_1+H_2$ value: if it is smaller than 1, the asymptotic decay behaves like $f^{-1-H_1-H_2}$ (for large $f$), whereas if $H_1+H_2\geq 1$, then the asymptotic decay is proportional to $f^{-2}$.

\begin{figure}[ht!]
    \centering
    \includegraphics[width=\textwidth]{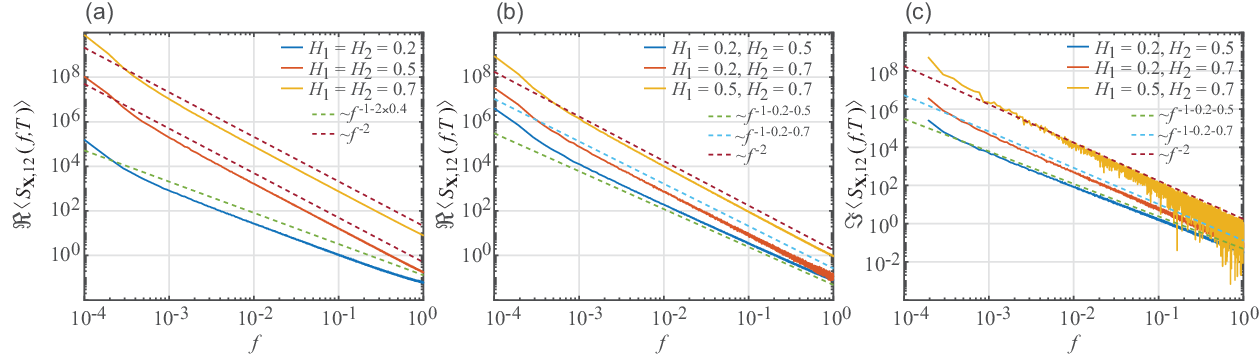}

    \caption{Ensemble-averaged (cross-)PSD. Solid lines correspond to PSD estimated from the simulated trajectories, dashed lines to their asymptotics. (a) $\Re \langle S_{\mathbf{X}, 12}(f, T)\rangle$, $H_1=H_2$; (b) $\Re \langle S_{\mathbf{X}, 12}(f, T)\rangle$, $H_1\neq H_2$; and (c) $\Im \langle S_{\mathbf{X}, 12}(f, T)\rangle$, $H_1\neq H_2$. \added[id=MB]{In each case, $\rho_{12}=0.5$.} Since the real part for both causal and well-balanced case is the same, we plot only only one corresponding line to each set of parameteres. Conversely, the imaginary part for the well-balanced case is 0, so we plot only $\Im\langle S_{\mathbf{X}, 12}\rangle$ for the causal 2D fBm. In all the panels, $T=2^{16}=65,536$ and the lines are based on $N=5,000$ trajectories.}
    \label{fig:psd1}
\end{figure}

\noindent Similarly to Figure \ref{fig:psd1}, in Figure \ref{fig:psd2} we present a comparison between estimated ensemble-averaged (cross-)PSD to their theoretical asymptotics for the increments of 2D fBm. Again, solid lines represent estimates from simulations, while dashed lines indicate the theoretical asymptotics. Different colored solid lines correspond to different sets of parameters $(H_1, H_2)$, whereas $\rho_{12}=0.5$ for all the cases.
Panels (a) and (b) present the real part, i.e., $\Re\langle S_{\Delta \mathbf{X}, 12}(f, T)\rangle$, with panel (a) corresponding to the case with $H_1=H_2$ (i.e., when the causal and well-balanced 2D fBms coincide), while panel (b) represents the causal 2D fBm increments with $H_1\neq H_2$. 
Panel (c) showcases the imaginary part, i.e., $\Im\langle S_{\Delta \mathbf{X}, 12}(f, T)\rangle$ for the sets $H_1\neq H_2$ corresponding to panel (b) (as for the $H_1=H_2$ case, the cross-PSD is 0).
\\
\noindent In all panels, we see that the asymptotic behavior of the ensemble-averaged (cross-)PSD for increments behaves in a similar way in all cases, regardless if $H_1+H_2 \gtrless 1$. The asymptotic behavior, as $f\to 0$ is proportional to $f^{1-H_1-H_2}$. Here, we see a quantitative difference of diverging or converging (cross-)PSD around 0 -- if $H_1+H_2>1$ then the (cross-)PSD diverges (in 1D, it corresponds to one of the definitions of long memory \cite{pipiras2017long}), while for $H_1+H_2<1$ the (cross-)PSD converges to 0.

\begin{figure}[ht!]
    \centering
    \includegraphics[width=\textwidth]{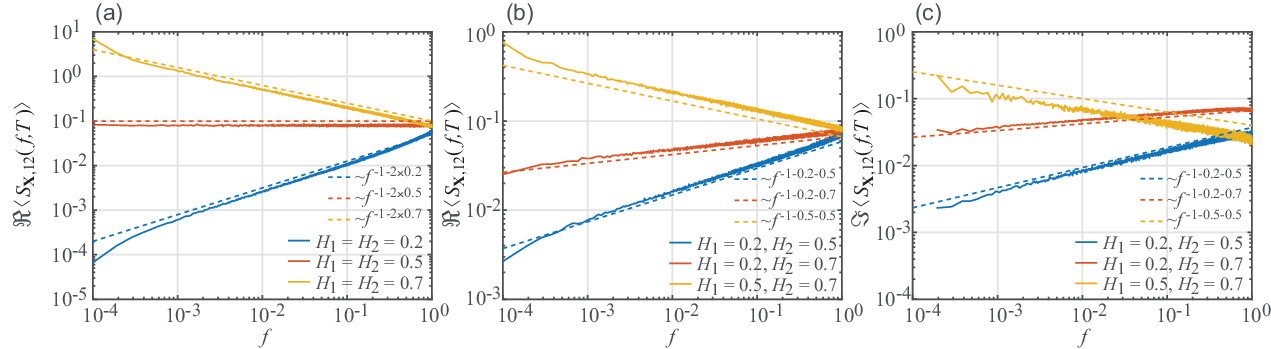}
    \caption{Ensemble-averaged (cross-)PSD for increments. Continuous lines correspond to PSD estimated from the simulated trajectories' increments, dashed lines to their asymptotics. (a) $\Re \langle S^\Delta_{ 12}(f, T)\rangle$, $H_1=H_2$; (b) $\Re \langle S^\Delta_{12}(f)\rangle$, $H_1\neq H_2$; and (c) $\Im \langle S^\Delta_{12}(f)\rangle$, $H_1\neq H_2$. \added[id=MB]{In each case, $\rho_{12}=0.5$. }Since the real part for both causal and well-balanced case is the same, we plot only only one corresponding line to each set of parameters. Conversely, the imaginary part for the well-balanced case is 0, so we plot only $\Im\langle S^\Delta_{12}(f)\rangle$ for the causal 2D fBm's increments. In all the panels, $T=2^{16}=65,536$ and the lines are based on $N=5,000$ trajectories.}
    \label{fig:psd2}
\end{figure}

\section{Conclusions}
\label{sec:conclusions}
In this work, we introduced a natural construction of 2D fBm that accounts for anisotropic scaling and cross-dimensional dependencies, extending the classical framework of Mandelbrot and van Ness to multivariate settings. By incorporating correlated Gaussian noises and a matrix-valued Hurst operator, we provided two distinct formulations, causal and well-balanced 2D fBm, and derived their respective auto- and cross-covariance across the two components. Our analysis highlights key differences between these two formulations, particularly in the asymmetry of causal fBm and implications for time-reversal. Furthermore, we analyzed the spectral properties of the processes, presenting both full analytical derivations and asymptotic behaviors in the frequency domain. Our results reveal how variations in the Hurst exponents and correlation structure manifest in the spectral content, offering new insights into the interpretation of multivariate time series in diverse systems. \added[id=MB]{
The results also clarify the qualitative distinction between the causal and well-balanced versions of the 2D fBm. While both share the same marginal scaling properties and stationary increments, their dependence structures differ} \added[id=DK]{substantially} \added[id=MB]{when $H_1 \neq H_2$. The causal model exhibits asymmetric cross-covariance structure and a phase shift in the spectral domain. In contrast, the well-balanced construction remains time-reversible for all parameter values, with symmetric cross-covariance and no cross-PSD structure. In this sense, the two formulations are complementary. The difference between the models provides a means to identify a model in a real data set.} 
A notable case arises when $H_1=H_2$. In this situation, the process remains a 2D fBm under arbitrary rotations. In contrast, when $H_1\neq H_2$, a rotation mixes the coordinates in such a way that the resulting marginals are no longer fBm, but instead they exhibit more complex dependencies. \added[id=DK]{A 2D processes with $H_1=H_2$ can thus describe a system with two-component fBm in the same geometric space, such as the two-dimensional motion of particles in a complex environment with underlying non-uniform structure. On the other hand, $H_1\neq H_2$ describes a more general case where each component may represent a different fBm, albeit exhibiting correlations. Such two-component processes can emerge in diverse systems. To name a few examples: in climate science, they can describe temperature \cite{syroka2001scaling,brody2002dynamicalpricing} and precipitation \cite{NessMandelbrot}; in hydrology, river flows \cite{NessMandelbrot,losada2020fractional} and erosion \cite{pelletier2007fractal}; and in physiology, heart rate \cite{fischer2003multi} and breathing patterns \cite{hoop1996fractal}.}

Extensive numerical simulations validated our theoretical findings, demonstrating consistency in both the time and frequency domains. The proposed framework thus provides a robust tool for modeling and analyzing complex anomalous diffusion processes in multidimensional environments, especially in settings where directional dependence and inter-component correlations cannot be ignored. This study opens several avenues for further exploration. A potential immediate extension involves generalizing the proposed constructions \added[id=MB]{and results} to higher dimensions. Our work opens the way to the application of 2D fBm to empirical data in fields such as neuroscience, single-particle tracking, structural dynamics, and financial econometrics, which may yield new insights into underlying spatiotemporal correlations. 


\section*{Acknowledgements}
\noindent This work was supported by NCN OPUS project No. 2024/53/B/HS4/00433 (to AW) and by the National Science Foundation (NSF) Grant 2102832 (to DK).

\appendix
\section{Proof of \autoref{thm:vfbm-cov}, causal case and some additional remarks}
\label{appendix:proof_causal_cov}
\begin{proof}
    Comparing the definition of the proposed causal 2D fBm process (cf. Definition \autoref{def:vfbm-causal}) to the  definition of operator fractional Brownian motion $B_H(t), t\geq 0,$ given in formula (9.3.22) from \cite{pipiras2017long}, 
\begin{align}
    \label{eq:vfbm_general}
    B_H(t) = \int_\mathbb{R} ((t-u)_+^D - (-u)_+^D) M_+ + ((t-u)_-^D - (-u)_-^D)M_-) B(\dd u),
\end{align}
we see that for our causal 2D fBm is equivalent when the matrices $M_+$ and $M_-$ are as follows 
\begin{align}
    M_+ &= \begin{bmatrix}
        \sigma_1 a_{H_1} & 0 \\
        \rho \sigma_2 a_{H_2} & \sqrt{1-\rho^2}\sigma_2 a_{H_2}
    \end{bmatrix},\\
    M_- &= 0 \cdot I_2,
\end{align}
for a size-two identity matrix $I_2$.
\\
\noindent Utilizing (9.3.28-29) from \cite{pipiras2017long}\added[id=MB]{i.e., }
\begin{align}
M_+ - M_- &= \sqrt{2\pi} \left[\sin \left(\frac{D\pi}{2}\right)\right]^{-1} \Gamma(D+I)^{-1} A_1,\label{eq:A_1_general}\\
M_+ + M_- &= \sqrt{2\pi} \left[\sin \left(\frac{D\pi}{2}\right)\right]^{-1} \Gamma(D+I)^{-1} A_2,\label{eq:A_2_general}
\end{align}
we can calculate $A = A_1 + i A_2$ that later will be used for determining the covariance structure of the proposed process
\begin{align}
    \label{eq:A_1}
    A_1 &= \frac{1}{\sqrt{2\pi}} \Gamma(D+I) \cdot \sin \left(\frac{D\pi}{2}\right) M_+,\\
    \label{eq:A_2}
    A_2 &= \frac{1}{\sqrt{2\pi}} \Gamma(D+I) \cdot \cos \left(\frac{D\pi}{2}\right) M_+,
\end{align}
where $D = H - \frac{1}{2}I$ is a 2x2 matrix, and all of the present functions ($\Gamma, \sin, \cos$) are understood as the primary matrix functions (def. 9.3.4 in \cite{pipiras2017long}). Therefore, we can explicitly write
\begin{align}
    A_1 &= \frac{1}{\sqrt{2\pi}} 
        \begin{bmatrix}
            \Gamma(H_1 + \frac{1}{2}) & 0 \\
            0 & \Gamma(H_2 + \frac{1}{2})
        \end{bmatrix}\cdot 
        \begin{bmatrix}
            \sin\left(\frac{(H_1 - \frac{1}{2})\pi}{2}\right) & 0 \\
            0 & \sin\left(\frac{(H_2 - \frac{1}{2})\pi}{2}\right)
        \end{bmatrix}\cdot \begin{bmatrix}
        \sigma_1 a_{H_1} & 0 \\
        \rho \sigma_2 a_{H_2} & \sqrt{1-\rho^2}\sigma_2 a_{H_2}
    \end{bmatrix} \nonumber \\
     & = \begin{bmatrix}
         \Gamma(H_1 + \frac{1}{2}) \sin\left(\frac{(H_1 - \frac{1}{2})\pi}{2}\right) \sigma_1 a_{H_1} & 0\\
         \rho\Gamma(H_2 + \frac{1}{2}) \sin\left(\frac{(H_2 - \frac{1}{2})\pi}{2}\right) \sigma_2 a_{H_2} & \sqrt{1-\rho^2}\Gamma(H_2 + \frac{1}{2}) \sin\left(\frac{(H_2 - \frac{1}{2})\pi}{2}\right) \sigma_2 a_{H_2}
     \end{bmatrix},
\end{align}
and similarly for $A_2$
\begin{align}
    A_2 & = \frac{1}{\sqrt{2\pi}} \begin{bmatrix}
         \Gamma(H_1 + \frac{1}{2}) \cos\left(\frac{(H_1 - \frac{1}{2})\pi}{2}\right) \sigma_1 a_{H_1} & 0\\
         \rho\Gamma(H_2 + \frac{1}{2}) \cos\left(\frac{(H_2 - \frac{1}{2})\pi}{2}\right) \sigma_2 a_{H_2} & \sqrt{1-\rho^2}\Gamma(H_2 + \frac{1}{2}) \cos\left(\frac{(H_2 - \frac{1}{2})\pi}{2}\right) \sigma_2 a_{H_2}
     \end{bmatrix}.
\end{align}
Thus, $A = A_1 + i A_2$ can be written as
\begin{align}
    \label{eq:Amatrix}
    A & = \frac{1}{\sqrt{2\pi}} \begin{bmatrix}
          \Gamma(H_1 + \frac{1}{2}) \exp\{i \frac{\pi}{2} (\frac{3}{2} - H_1)\} \sigma_1 a_{H_1} & 0\\
         \rho\Gamma(H_2 + \frac{1}{2}) \exp\{i \frac{\pi}{2} (\frac{3}{2} - H_2)\}\sigma_2 a_{H_2} & \sqrt{1-\rho^2}\Gamma(H_2 + \frac{1}{2}) \exp\{i \frac{\pi}{2} (\frac{3}{2} - H_2)\} \sigma_2 a_{H_2}
     \end{bmatrix}.
\end{align}
\noindent In order to obtain the covariance structure, we calculate $C =[c_{jk}]_{j,k=1,2} = A A^*$ (from Proposition 9.3.19 \cite{pipiras2017long}). In our case, its elements are given by
\begin{align}
    \begin{cases}
    c_{jj}&=\frac{1}{2\pi} \sigma^2_j \Gamma^2(H_j+0.5)a^2_{H_j},\\
    c_{jk} &= \frac{1}{2\pi} \rho\sigma_j\sigma_k \Gamma(H_j+0.5)\Gamma(H_k+0.5)a_{H_j}a_{H_k}e^{-i\frac{\pi}{2}(H_j-H_k)},  \quad j\neq k,
    \end{cases}
    \label{eq:c_jk_causal}
\end{align}
for $j, k =1, 2$. Note that the matrix $C$ is Hermitian, $C=C^*$, and thus its diagonal elements are real.
According to Proposition 9.3.19 \cite{pipiras2017long}, the covariance structure is

\begin{align}
    \gamma_{jk}(t, s) \equiv \langle X_j(t)X_k(s)\rangle = \frac{\sigma_j\sigma_k}{2}\left(w_{jk}(t) |t|^{H_j + H_k} + w_{jk}(-s) |s|^{H_j+H_k} - w_{jk}(t-s)|t-s|^{H_j+H_k} \right), 
\end{align}
where $\sigma_j^2 = \langle X_j^2(1)\rangle$ and
\begin{align}
w_{jk}(u) = \begin{cases}
    \rho_{jk} - \eta_{jk} \sign(u), &\quad H_j+H_k \neq 1,\\
    \rho_{jk} - \eta_{jk} \sign(u)\log|u|, &\quad H_j+H_k =1.
\end{cases}    
\end{align}
\end{proof}
\noindent The connection between the elements of matrix $C$ given in Eq. \eqref{eq:c_jk_causal} and the parameters $\rho_{jk}$ and $\eta_{jk}$ are as follows
\begin{align}
    \sigma_j \sigma_k \rho_{jk} = 4 b_1\left( \frac{H_j +H_k}{2}\right) \Re(c_{jk}),\\
    \sigma_j \sigma_k \eta_{jk} = 4 b_2\left( \frac{H_j +H_k}{2}\right) \Im(c_{jk}),
\end{align}
where 
\begin{align}
    b_1(H) &= \begin{cases}
        \frac{\Gamma(2-2H) \cos(H\pi)}{2H (1-2H)}, &\quad H \neq \frac{1}{2},\\
        \frac{\pi}{2}, &\quad H = \frac{1}{2},
    \end{cases}\\
    b_2(H) &= \frac{\Gamma(2-2H)\sin(H\pi)}{2H(1-2H)}.
\end{align}
Let us calculate $\rho_{jk}$ explicitly. First, consider $j=k$, and assume that $H_j\neq \frac{1}{2}$. We have
\begin{align}
    \sigma_j \sigma_j \rho_{jj} &= 4 b_1\left( \frac{H_j +H_j}{2}\right) \Re(c_{jj}) = 4 b_1(H_j) \Re(c_{jj}) \nonumber \\
    &= 4\frac{\Gamma(2-2H_j) \cos(H_j\pi)}{2H_j (1-2H_j)} \frac{1}{2\pi} \Gamma^2\left(H_j + \frac{1}{2}\right) \sigma^2_j a^2_{H_j}  \nonumber \\
     & = 4\sigma^2_j \frac{\overbrace{\Gamma(1-2H_j)}^{\Gamma(z)\Gamma(1-z) = \frac{\pi}{\sin (\pi z)}} \cos(H_j\pi)}{2H_j} \frac{1}{2\pi} \Gamma^2\left(H_j + \frac{1}{2}\right) a^2_{H_j} \nonumber \\
     & = 2\sigma_j^2 \frac{\cos(H_j\pi)}{\underbrace{2H_j \Gamma(2H_j)}_{\Gamma(2H_j +1)}\sin(2H_j \pi)} \Gamma^2\left(H_j + \frac{1}{2}\right) \underbrace{a^2_{H_j}}_{=\frac{\Gamma(2H_j+1)\sin(H_j\pi)}{\Gamma^2(H_j + \frac{1}{2})}} \nonumber \\
     & = \sigma_j^2.
\end{align}
Therefore, $\rho_{jj}=1$ as expected from the correlation coefficient between $X_j(1)$ and $X_j(1)$.

\noindent Now, let us deal with the more interesting case, i.e., $\rho_{12}$. We have the following
{\small 
\begin{align}
    \sigma_1 \sigma_2 \rho_{12} &= 4 b_1\left( \frac{H_1 +H_2}{2}\right) \Re(c_{12}) \nonumber \\
    &= 4\frac{\Gamma(2-(H_1+H_2)) \cos\left(\frac{H_1+H_2}{2}\pi\right)}{(H_1+H_2) (1-(H_1+H_2))} \frac{1}{2\pi} \Gamma\left(H_1 + \frac{1}{2}\right)\Gamma\left(H_2 + \frac{1}{2}\right) \rho\sigma_1\sigma_2 a_{H_1}a_{H_2} \cos\left(\frac{(H_2-H_1)\pi}{2}\right)  \nonumber \\
     & = 4\rho\sigma_1\sigma_2 \frac{\overbrace{\Gamma(1-(H_1+H_2))}^{\Gamma(z)\Gamma(1-z) = \frac{\pi}{\sin (\pi z)}} \cos\left(\frac{H_1+H_2}{2}\pi\right)}{H_1+H_2} \frac{1}{2\pi} \Gamma\left(H_1 + \frac{1}{2}\right)\Gamma\left(H_2 + \frac{1}{2}\right) a_{H_1}a_{H_2} \cos\left(\frac{(H_2-H_1)\pi}{2}\right) \nonumber \\
     & = 2\rho\sigma_1\sigma_2 \frac{\pi}{\sin((H_1+H_2)\pi) \Gamma(H_1+H_2)} \frac{\cos\left(\frac{H_1+H_2}{2}\pi\right)}{H_1+H_2} \frac{1}{\pi} \Gamma\left(H_1 + \frac{1}{2}\right)\Gamma\left(H_2 + \frac{1}{2}\right) \underbrace{a_{H_1}}_{=\frac{\sqrt{\Gamma(2H_1+1)\sin(H_1\pi)}}{\Gamma(H_1 + \frac{1}{2})}}a_{H_2}\cos\left(\frac{(H_2-H_1)\pi}{2}\right)  \nonumber \\
     & = \rho \sigma_1 \sigma_2 \frac{\sqrt{\Gamma(2H_1+1)\Gamma(2H_2+1)\sin(H_1\pi)\sin(H_2\pi)}}{\Gamma(H_1+H_2+1)\sin\left(\frac{H_1+H_2}{2}\pi\right) }\cos\left(\frac{(H_2-H_1)\pi}{2}\right).
\end{align}
}
Thus, the cross-correlation $\rho_{12}$ is given by 
\begin{align}
    \label{eq:rhovsrho12}
    \rho_{12} = \rho\frac{\sqrt{\Gamma(2H_1+1)\Gamma(2H_2+1)\sin(H_1\pi)\sin(H_2\pi)}}{\Gamma(H_1+H_2+1)\sin\left(\frac{H_1+H_2}{2}\pi\right) }\cos\left(\frac{(H_2-H_1)\pi}{2}\right).
\end{align}
Note that, for the special case $H_1=H_2$, the cross-correlation $\rho_{12}$ is equal to the correlation between noises, i.e., $\rho_{12}=\rho$. 
Lastly, to have the full covariance structure, we need to calculate coefficients $\eta_{jk}$. First, as $\Im(c_{jj}) = 0$ ($C$ being a Hermitian matrix), we have $\eta_{jj}=0 \text{ for } j=1,2$. The more interesting case is $j\neq k$, which after similar calculations as for $\rho_{12}$ leads to 
\begin{align}
    \label{eq:rhovseta12}
    \eta_{12} = \rho\frac{\sqrt{\Gamma(2H_1+1)\Gamma(2H_2+1)\sin(H_1\pi)\sin(H_2\pi)}}{\Gamma(H_1+H_2+1)\cos\left(\frac{H_1+H_2}{2}\pi\right) }\sin\left(\frac{(H_2-H_1)\pi}{2}\right).
\end{align}
It can be expressed as
\begin{align}
    \label{eq:eta_vs_rho}
    \eta_{12} = \rho_{12} \tan\left(\pi \frac{H_2-H_1}{2}\right)\tan\left(\pi \frac{H_1+H_2}{2}\right) = -\rho_{12}\tan\left(\pi \frac{H_1-H_2}{2}\right)\tan\left(\pi \frac{H_1+H_2}{2}\right),
\end{align}
which coincides with the results for the causal process given in Section 3.3 of \cite{amblard2010basic}.

\section{Construction of a well-balanced 2D fBm}\label{App:B}
\noindent First, let us consider a one-dimensional well-balanced fractional Brownian motion given  in Definition \autoref{def:wb-fbm1d} below.
\begin{definition}[Well-balanced 1D fBm]
\label{def:wb-fbm1d}
Well-balanced fractional Brownian motion $X^*(t), t\geq 0,$ is given by the following time-domain representation
    \begin{align}
    X^*(t)      = a^*_H \int_\mathbb{R} ((t-u)_+^{H-\frac{1}{2}} - (-u)_+^{H-\frac{1}{2}}) + ((t-u)_-^{H-\frac{1}{2}} - (-u)_-^{H-\frac{1}{2}})) B(\dd u),
\end{align}
where $H \in (0,1)$ and $a^*_H>0$. 
\end{definition}

\begin{theorem}
\label{thm:well-balanced-constant}
The constant $a^*_H$ for which $\langle X^{*2}(t) \rangle = 1$ for the well-balanced 1D fBm is given by
    \begin{align*}
    a^{*2}_H &= \frac{2H(1-2H)\pi}{8\Gamma(2-2H)\cos(H\pi) \Gamma^2(H+\frac{1}{2}) \cos^2\left(\frac{\pi (H-\frac{1}{2})}{2}\right)}.
\end{align*}
\begin{proof}
  
Let us recall the definition of the function given in Eq. (\ref{eq:f_pm})
\begin{align*}
    f_\pm(x; t, \beta) = (t-x)_\pm^\beta - (-x)_\pm^\beta.
\end{align*}
\noindent In order to calculate the autocovariance of the process, let us first utilize It\^o's lemma. We have
\begin{align}
    \langle X^*(t) X^*(s) \rangle &= a^{*2}_H \int_\mathbb{R} (f_+(u; t, D) + f_-(u;t, D)) (f_+(u; s, D) + f_-(u; s, D)) \dd u \nonumber \\
    & \stackrel{\textrm{Plancherel}}{=}\frac{a^{*2}_H}{2\pi} \int_\R (\widehat{f_+}(\xi; t, D) + \widehat{f_-}(\xi;t, D)) \overline{(\widehat{f_+}(\xi; s, D) + \widehat{f_-}(\xi; s, D))} \dd \xi,
\end{align}
where $\overline{z}$ is a complex conjugate of $z$ and $\widehat{f_\pm}(\xi; t, D)$ is a Fourier transform of $f_\pm$. Thus, it is given by \cite{gradshteyn2014table} (formulae 3.76.4 and 3.76.9)
\begin{align}
    \widehat{f_\pm}(\xi; t, D) \equiv \int_\mathbb{R} e^{i u \xi} f(u;t, D)\dd u= \frac{e^{i t x} - 1}{i x |x|^D} \Gamma(D+1) e^{\mp i \sign(x) \frac{\pi D}{2}}
\end{align}
and
\begin{align}
    \widehat{f_+}(\xi; t, D) + \widehat{f_-}(\xi; t, D) =  \frac{e^{i t x} - 1}{i x |x|^D} \Gamma(D+1) 2 \cos^2\left(\frac{\pi D}{2}\right).
\end{align}
Consequently, the autocovariance function becomes
\begin{align}
    \langle X^*(t) X^*(s) \rangle & = \frac{a^{*2}_H}{2\pi} \int_\mathbb{R} \frac{e^{i t x} - 1}{i x |x|^D} \frac{e^{-i s x} - 1}{-i x |x|^D}  \Gamma^2(D+1) 4 \cos^2\frac{\pi D}{2} \dd x= \nonumber\\
    & = 4 \frac{a^{*2}_H}{2\pi}\Gamma^2(D+1)  \cos^2\left(\frac{\pi D}{2}\right) \int_\R \frac{(e^{i t x} - 1)(e^{-isx}-1)}{|x|^{2(D+1)}} \dd x.
\end{align}
\noindent When $t=s=1$ we have
\begin{align}
    \langle X^{*2}(1) \rangle & = 4 \frac{a^{*2}_H}{2\pi}\Gamma^2(D+1)  \cos^2\left(\frac{\pi D}{2}\right) \int_\R \frac{(e^{i x} - 1)(e^{-ix}-1)}{|x|^{2(D+1)}} \dd x  \nonumber\\
    & = \frac{4a^{*2}_H}{\pi}\Gamma^2(D+1)  \cos^2\left(\frac{\pi D}{2}\right) \int_\R \frac{1-\cos x}{|x|^{2(D+1)}} \dd x \nonumber \\
    & = \frac{a^{*2}_H}{\pi} \Gamma^2(D+1) \cos^2\left(\frac{\pi D}{2}\right) \frac{4\pi}{\Gamma(2(D+1)) \sin((D+\frac{1}{2})\pi)}  \nonumber\\
    & = \frac{a^{*2}_H}{\pi} \Gamma^2(H+\frac{1}{2}) \cos^2\left(\frac{\pi (H-\frac{1}{2})}{2}\right) \frac{4\pi}{\Gamma(2H+1) \sin(H\pi)}.
\end{align}
\noindent To ensure that the variance of the process $X^*(t)$ is equal to 1 at time $t=1$, the constant $a_H$ has to be as follows 
\begin{align}
    a^{*2}_H &= \frac{2H \Gamma(2H) \sin(\pi H)}{4 \Gamma^2(H+\frac{1}{2}) \cos^2 \left(\frac{\pi (H-\frac{1}{2})}{2}\right)} \nonumber \\
    & = \frac{2H(1-2H)\pi}{8\Gamma(2-2H)\cos(H\pi) \Gamma^2(H+\frac{1}{2}) \cos^2\left(\frac{\pi (H-\frac{1}{2})}{2}\right)}.
\end{align}
\end{proof}
\end{theorem}
     
\section{Proof of \autoref{thm:vfbm-cov}, well-balanced case and remarks}
\label{appendix:proof_well-balanced_cov}
    \begin{proof}
        Again, looking at the general form given in Eq. (\ref{eq:vfbm_general}), we can calculate matrices $M_+, M_-$ and then, $A_1, A_2$ and $C$.
It is clear that matrices $M_+$ and $M_-$ are the same, and given by
\begin{align}
    M_\pm = \begin{bmatrix}
        a^*_{H_1} \sigma_1 & 0 \\
        a^*_{H_2}\rho \sigma_2 & a^*_{H_2}\sqrt{1-\rho^2}\sigma_2
    \end{bmatrix}.
\end{align}
Thus, by the relations given in Eqs. \eqref{eq:A_1_general}-\eqref{eq:A_2_general} we have
\begin{align}
    A_1 &= 0\cdot I,\\
    A_2 &= \sqrt{\frac{2}{\pi}} \cos\frac{D\pi}{2} \Gamma(D+I) M_+ = \begin{bmatrix}
        a^*_{H_1} \sigma_1 \cos\left(\frac{(H_1-\frac{1}{2})\pi}{2}\right) \Gamma\left(H_1+\frac{1}{2}\right)& 0 \\
        a^*_{H_2}\rho \sigma_2 \cos\left(\frac{(H_2-\frac{1}{2})\pi}{2}\right) \Gamma\left(H_2+\frac{1}{2}\right) & a^*_{H_2}\sqrt{1-\rho^2}\sigma_2\cos\left(\frac{(H_2-\frac{1}{2})\pi}{2}\right) \Gamma\left(H_2+\frac{1}{2}\right)
    \end{bmatrix}.
\end{align}

\noindent In the case of well-balanced 2D fBm, $A = A_1 + i A_2$ can be written as
\begin{align}
    \label{eq:Amatrix_wb}
    A & = iA_2 = i\begin{bmatrix}
        a^*_{H_1} \sigma_1 \cos\left(\frac{(H_1-\frac{1}{2})\pi}{2}\right) \Gamma\left(H_1+\frac{1}{2}\right)& 0 \\
        a^*_{H_2}\rho \sigma_2 \cos\left(\frac{(H_2-\frac{1}{2})\pi}{2}\right) \Gamma\left(H_2+\frac{1}{2}\right) & a^*_{H_2}\sqrt{1-\rho^2}\sigma_2\cos\left(\frac{(H_2-\frac{1}{2})\pi}{2}\right) \Gamma\left(H_2+\frac{1}{2}\right)
    \end{bmatrix}.
\end{align}
Since $A=A_1 + iA_2 = iA_2$ and $C = A A^*$, elements of $C$ are given by 
\begin{align}
    \begin{cases}
    c_{jj}&=\frac{2}{\pi} \sigma^2_j \cos^2\left(\frac{(H_j-0.5)\pi}{2}\right)\Gamma^2(H_j+0.5)a^{*2}_{H_j},\\
    c_{jk} &= \frac{2}{\pi} \rho\sigma_j\sigma_k\cos\left(\frac{(H_j-0.5)\pi}{2}\right)\cos\left(\frac{(H_k-0.5)\pi}{2}\right) \Gamma(H_j+0.5)\Gamma(H_k+0.5)a^*_{H_j}a^*_{H_k}, \quad j\neq k,
    \end{cases}
    \label{eq:c_jk_wb}
\end{align}
for $j, k = 1,2$. We note that $C$ is a real matrix, in contrast to its equivalent in the causal 2D fBm. The consequence of that is that $\eta_{jk} = 0$ for all $j,k=1,2$, and therefore this is a time-reversible version of a 2D fBm. Thus, the covariance structure of two-dimensional well-balanced 2D fBm simplifies to
\begin{align}
    \langle X^*_j(t)X^*_k(s)\rangle = \frac{\sigma_j\sigma_k\rho_{jk}}{2}\left(|t|^{H_j + H_k} + |s|^{H_j+H_k} - |t-s|^{H_j+H_k} \right).
\end{align}
\end{proof} 

\noindent We can also calculate coefficients $\rho_{jk}$ present in the covariance structure
{\small \begin{align}
    \sigma_j^2 \rho_{jj} &= 4 b_1(H_j) \Re(c_{jj}) = 4\underbrace{ \frac{\Gamma(2-2H_j) \cos(H_j \pi)}{2H_j (1-2H_j)}}_{b_1(H_j)} \underbrace{\frac{2}{\pi} \sigma_j^2 a_{H_j}^2\Gamma^2\left(H_j+\frac{1}{2}\right)\cos^2\left(\frac{(H_j-\frac{1}{2})\pi}{2}\right) }_{c_{jj}}  \nonumber\\
    & = 4 \frac{\Gamma(2-2H_j) \cos(H_j \pi)}{2H_j (1-2H_j)} \frac{2}{\pi} \sigma_j^2 \Gamma\left(H_j+\frac{1}{2}\right)\cos^2\left(\frac{(H_j-\frac{1}{2})\pi}{2}\right) \frac{2H_j(1-2H_j)\pi}{8\Gamma(2-2H_j)\cos(\pi H_j) \Gamma^2\left(H_j+\frac{1}{2}\right)\cos^2\left(\frac{(H_j-\frac{1}{2})\pi}{2}\right)}=\sigma_j^2.
\end{align}}
As expected $\rho_{jj}=1$, since it is a correlation of $X^*_j(1)$ with itself. For more interesting case, i.e., for $\rho_{12}$, we have 
{\scriptsize
\begin{align}
    \sigma_1 \sigma_2 \rho_{12} & = 4 b_1\left(\frac{H_1+H_2}{2}\right) \Re(c_{12}) =\nonumber \\
    &= 4 \underbrace{\frac{\Gamma(2-(H_1+H_2)) \cos \frac{(H_1+H_2)\pi}{2}}{(H_1+H_2)(1-(H_1+H_2))}}_{b_1\left(\frac{H_1+H_2}{2}\right)}\underbrace{\frac{2}{\pi} \cos\left(\frac{(H_1-\frac{1}{2})\pi}{2}\right)\Gamma\left(H_1+\frac{1}{2}\right)\cos\left(\frac{(H_2-\frac{1}{2})\pi}{2}\right)\Gamma\left(H_2+\frac{1}{2}\right) \sigma_1\sigma_2 a^*_{H_1}a^*_{H_2}\rho}_{\Re(c_{12})}  \nonumber\\
    & = \frac{8}{\pi}\sigma_1\sigma_2\rho \frac{\Gamma(2-(H_1+H_2)) \cos \frac{(H_1+H_2)\pi}{2}}{(H_1+H_2)(1-(H_1+H_2))} \frac{\cos\left(\frac{(H_1-\frac{1}{2})\pi}{2}\right)\Gamma\left(H_1+\frac{1}{2}\right)\cos\left(\frac{(H_2-\frac{1}{2})\pi}{2}\right)\Gamma\left(H_2+\frac{1}{2}\right) \sqrt{2H_1(1-2H_1)}\sqrt{2H_2(1-2H_2)}\pi}{8\sqrt{\Gamma(2-2H_1)\cos(\pi H_1)}\sqrt{\Gamma(2-2H_2)\cos(\pi H_2)}\Gamma\left(H_1+\frac{1}{2}\right)\cos\left(\frac{(H_1-\frac{1}{2})\pi}{2}\right)\Gamma\left(H_2+\frac{1}{2}\right)\cos\left(\frac{(H_2-\frac{1}{2})\pi}{2}\right)} \nonumber\\
    & =\sigma_1\sigma_2\rho \frac{\cos\left(\frac{(H_1+H_2)\pi}{2}\right)}{\Gamma(H_1+H_2+1)\sin((H_1+H_2)\pi)} \frac{\sqrt{\Gamma(2H_1+1)\Gamma(2H_2+1)\sin(2H_1\pi)\sin(2H_2\pi)}}{\sqrt{\cos(H_1\pi)\cos(H_2\pi)}}  \nonumber\\
    & = \sigma_1\sigma_2\rho \frac{\sqrt{\Gamma(2H_1+1)\Gamma(2H_2+1)}}{\Gamma(H_1+H_2+1)} \frac{\sqrt{\sin(H_1\pi)\sin(H_2\pi)}}{\sin\left(\frac{(H_1+H_2)\pi}{2}\right)}.
\end{align}
}
And thus, 
\small 
\begin{align}
    \rho_{12} = \rho \frac{\sqrt{\Gamma(2H_1+1)\Gamma(2H_2+1)}}{\Gamma(H_1+H_2+1)} \frac{\sqrt{\sin(H_1\pi)\sin(H_2\pi)}}{\sin\left(\frac{(H_1+H_2)\pi}{2}\right)},
\end{align}
which differs from the one given in the causal case missing a $\cos\left(\frac{(H_2-H_1)\pi}{2}\right)$ factor. Similarly, the condition that 
\begin{align}
\rho_{12}^2 \leq \left(\frac{\sqrt{\Gamma(2H_1+1)\Gamma(2H_2+1)}}{\Gamma(H_1+H_2+1)} \frac{\sqrt{\sin(H_1\pi)\sin(H_2\pi)}}{\sin\left(\frac{(H_1+H_2)\pi}{2}\right)} \right)^2
\end{align}
is the one present in the literature \cite{pipiras2017long} that guarantees the proper covariance structure.

\section{Proof of \autoref{thm:spectral_vfbm_inc}, spectral content of the increment process of 2D fBm}
\label{appendix:proof_psd_inc}
\begin{proof}
    First, following \cite{pipiras2017long}, we know the spectral representation of the increments
    \begin{align}
        \Delta \mathbf{Z}(n) \equiv \mathbf{Z}(n+1)-\mathbf{Z}(n) = \int_\mathbb{R} \frac{e^{i(n+1)x} - e^{inx}}{ix}\left(x_+^{-D} A + x_-^{-D}\overline{A}\right) \widehat B(\dd x),
    \end{align}
    where $A$ is given in Eq. \eqref{eq:Amatrix}, $D = H-\frac{1}{2}I$ and $\tilde B$ is Gaussian spectral measure with $\left\langle|\widehat B(\dd x)|^2\right\rangle=\dd x$ control measure. 
    Now we can write 
    \begin{align*}
        \left\langle\Delta\mathbf{Z}(k)\Delta\mathbf{Z}(0)\right\rangle & = \left\langle\int_\mathbb{R}\int_\mathbb{R} \frac{e^{i(k+1)x} - e^{ikx}}{ix} \overline{\left(\frac{e^{iy} - 1}{iy}\right)}\left(x_+^{-D} A + x_-^{-D}\overline{A}\right)\cdot \left(y_+^{-D} A + y_-^{-D}\overline{A}\right)^* \widehat{B}(\dd x)\widehat B^*(\dd y)\right\rangle\\
        &= \int_\mathbb{R} \frac{e^{i(k+1)x} - e^{ikx}}{ix} \overline{\left(\frac{e^{ix} - 1}{ix}\right)}\left(x_+^{-D} A + x_-^{-D}\overline{A}\right)\cdot \left(x_+^{-D} A + x_-^{-D}\overline{A}\right)^* \dd x\\
        & = \sum_{n=-\infty}^\infty \int_{-\pi+2n\pi }^{\pi+2n\pi}\frac{e^{i(k+1)x} - e^{ikx}}{ix} \overline{\frac{e^{ix} - 1}{ix}}\left(x_+^{-D} A + x_-^{-D}\overline{A}\right)\cdot \left(x_+^{-D} A + x_-^{-D}\overline{A}\right)^* \dd x\\
        &\textrm{substituging } x =f+2n\pi \textrm{ and exchanging the integral and sum we obtain}
        \\
        & = \int_{-\pi}^\pi \sum_{n=-\infty}^\infty e^{i k (f + 2 n \pi)} \frac{|e^{i(f + 2n\pi)}-1|^2}{(f+2n\pi)^2} \left((f+2n\pi)_+^{-D} A + (f+2n\pi)_-^{-D}\overline{A}\right)\cdot \left((f+2n\pi)_+^{-D} A + (f+2n\pi)_-^{-D}\overline{A}\right)^* \dd f\\
        &= \int_{-\pi}^\pi e^{i k f} |e^{if}-1|^2 \sum_{n=-\infty}^\infty \left((f+2n\pi)_+^{-D} A A^*(f+2n\pi)_+^{-D}+ (f+2n\pi)_-^{-D}\overline{A}\overline{A^*}(f+2n\pi)_-\right)/(f+2n\pi)^2 \dd f.
    \end{align*}
    Remembering that $AA^*=C$ (elements of $C$ are given in Eqs. \eqref{eq:c_jk_causal} and \eqref{eq:c_jk_wb}), and taking Remark 1.3.6 and 1.3.8 \cite{pipiras2017long} we obtain the final result for the power spectral density of $\Delta \mathbf{Z}$.

\noindent     Considering element-wise power spectral density, we have
    \begin{align*}
        S_{\Delta \mathbf{Z}, jk}(f) &= |e^{if}-1|^2 \sum_{n=-\infty}^\infty \left[(f + 2n\pi)_+^{-(d_j+d_k)} c_{jk} + (f+2n\pi)_-^{-(d_j+d_k)}\overline{c_{jk}}\right]/(f+2n\pi)^2 \\
        &= |e^{if}-1|^2 \sum_{n=-\infty}^\infty \left[(f + 2n\pi)_+^{-(H_j+H_k+1)} c_{jk} + (f+2n\pi)_-^{-(H_j+H_k+1)}\overline{c_{jk}}\right]/(f+2n\pi)^2,
    \end{align*}
    as $d_j = H_j-1/2$.

\noindent     Moreover, by simple calculations we have the asymptotics as $f \to 0$ 
    \begin{align*}
        S_{\Delta \mathbf{Z}, jk}(f) \sim c_{jk}f^{(-H_j-H_k+1)}.
    \end{align*}
\end{proof}

\section{Proof of \autoref{thm:spectral_vfbm}, spectral content of 2D fBm trajectory}
\label{appendix:proof_psd}

\begin{proof}
Power spectral density for the trajectory $\mathbf{Z}(t), 0\leq t\leq T$, is given by
\begin{align}
    S_{jk}(f, T) \equiv \frac{1}{T} \int_0^T e^{i f t} Z_j(t)\, \dd t \overline{\int_0^T e^{ifs} Z_k(s)\,\dd s}.
\end{align}
It is a random variable depending both on the considered frequency $f$, time horizon $T$, and the trajectory $\mathbf{Z}(t)$.
Since we want to consider the ensemble-averaged power spectral density, we have
\begin{align*}
    \langle S_{jk}(f, T)\rangle &= \frac{1}{T} \left\langle \int_0^T e^{i f t} Z_j(t)\, \dd t \overline{\int_0^T e^{ifs} Z_k(s)\,\dd s}\right\rangle = \ |\textrm{by Fubini theorem}| \\
    &= \frac{1}{T}  \int_0^T \int_0^T e^{i f (t-s)} \left\langle Z_j(t)Z_k(s)\right\rangle\, \dd t \, \dd s \\
    &= \frac{1}{T}  \int_0^T \int_0^T e^{i f (t-s)} \gamma_{jk}(t, s)\, \dd t \, \dd s,
\end{align*}
where $\gamma_{jk}(t, s)$ is the covariance function for $Z_j$ and $Z_k$. For simplicity, we write the general form of this function with weighting function $w_{jk}$ that has explicit form for both causal and well-balanced cases (Eq. \eqref{eq:w_jk_causal}).
\\
\noindent Here, we consider the case when $H_j+H_k\neq 1$ or the considered model is well-balanced 2D fBm (so that, there is no 
$\log$ part in the cross covariance function).
By rescaling the integrals' variables $\frac{t}{T}\rightarrow  x, \frac{s}{T}\rightarrow  y$ and by introducing $\tilde{\omega} = f T$, we have
\begin{align}
    \langle S_{jk}(f, T)\rangle &= T^{H_j +H_k + 1}\frac{\sigma_j \sigma_k}{2} \int_0^1 \int_0^1 e^{i \tilde{\omega} (x-y)} \left(w_{jk}(x) x^{H_j+H_k} + w_{jk}(-y) y^{H_j+H_k} - w_{jk}(x-y) |x-y|^{H_j+H_k}\right)\,\dd x\,\dd y
\end{align}
\noindent We calculate each element separately. Let us introduce the notation
\begin{align}
    \mathbb{I} &\equiv \int_0^1 \int_0^1 e^{i\tilde{\omega}(x-y)} w_{jk}(x) x^{H_j+H_k} \dd y\, \dd x, \\
    \mathbb{II} &\equiv \int_0^1 \int_0^1 e^{i\tilde{\omega}(x-y)} w_{jk}
    (-y) y^{H_j+H_k} \dd y\, \dd x,\\
    \mathbb{III} &\equiv \int_0^1 \int_0^1 e^{i\tilde{\omega}(x-y)} w_{jk}(x-y) |x-y|^{H_j+H_k} \dd y\, \dd x.
\end{align}
Then the resulting integral is equal to 
\begin{align}
    \langle S_{jk}(f, T)\rangle &= T^{H_j +H_k + 1}\frac{\sigma_j \sigma_k}{2} (\mathbb{I}+\mathbb{II}-\mathbb{III}).
\end{align}
Let us start with $\mathbb{I}$
\begin{align*}
    \mathbb{I} &\equiv \int_0^1 \int_0^1 e^{i\tilde{\omega}(x-y)} w_{jk}(x) x^{H_j+H_k} \dd y\, \dd x \\
    &= \int_0^1 e^{i\tilde{\omega} x} \frac{1-e^{-i \tilde{\omega}}}{i \tilde{\omega}} x^{H_j+H_k} \underbrace{w_{jk}(x)}_{\text{constant } \forall x>0}\,\dd x  \\
    &= w_{jk}(1) \frac{1-e^{-i \tilde{\omega}}}{i \tilde{\omega}} \left[\int_0^1 \cos(\tilde{\omega} x) + i \sin(\tilde{\omega} x))x^{H_j +H_k}\,\dd x \right].
\end{align*}
Now, let us consider $\mathbb{II}$
\begin{align*}
    \mathbb{II} &\equiv \int_0^1 \int_0^1 e^{i\tilde{\omega}(x-y)} w_{jk}(-y) y^{H_j+H_k} \dd x\, \dd y \\
    &= \int_0^1 e^{-i\tilde{\omega} y} \frac{e^{i \tilde{\omega}}-1}{i \tilde{\omega}} y^{H_j+H_k} \underbrace{w_{jk}(-y)}_{\text{constant } \forall y>0}\,\dd y  \\
    &= w_{jk}(-1) \frac{e^{i \tilde{\omega}}-1}{i \tilde{\omega}} \left[\int_0^1 \cos(\tilde{\omega} x) - i \sin(\tilde{\omega} x))y^{H_j +H_k}\,\dd x \right].
\end{align*}
Summing up the first two parts, we have
\begin{align}
    \mathbb{I}+\mathbb{II} =& \int_0^1 \cos(\tilde{\omega} x)x^{H_j+H_k}\,\dd x\left[w_{jk}(1)\frac{1-e^{-i\tilde{\omega}}}{i\tilde{\omega}} - w_{jk}(-1)\frac{1-e^{i \tilde{\omega}}}{i \tilde{\omega}} \right]\\
    &+\int_0^1 \sin(\tilde{\omega} y)y^{H_j+H_k}\,\dd y\left[w_{jk}(1)\frac{1-e^{-i\tilde{\omega}}}{i\tilde{\omega}} + w_{jk}(-1)\frac{1-e^{i \tilde{\omega}}}{i \tilde{\omega}}\right].
\end{align}
Reintroducing functions $\mathcal{C}_{jk}$ and $\mathcal{S}_{jk}$ as
\begin{align}
    \mathcal{C}_{jk} \equiv \mathcal{C}_{jk}(\tilde{\omega})&= \int_0^1 \cos(\tilde{\omega} x) x^{H_j+H_k} \,\dd x, \\
    \mathcal{S}_{jk} \equiv \mathcal{S}_{jk}(\tilde{\omega})&= \int_0^1 \sin(\tilde{\omega} x) x^{H_j+H_k} \,\dd x,
\end{align}
and utilizing that $w_{jk}(x) = \rho_{jk}-\sign(x)\eta_{jk}$ we can simplify the expression
\begin{align}
    \mathbb{I}+\mathbb{II} =& \mathcal{C}_{jk}\left[(\rho_{jk}-\eta_{jk})\frac{1-\cos(\tilde{\omega})+i\sin(\tilde{\omega})}{i\tilde{\omega}} - (\rho_{jk}+\eta_{jk})\frac{1-\cos(\tilde{\omega})-i\sin(\tilde{\omega})}{i \tilde{\omega}} \right]\nonumber \\
    &+ \mathcal{S}_{jk}\left[(\rho_{jk}-\eta_{jk})\frac{1-\cos(\tilde{\omega})+i\sin(\tilde{\omega})}{i\tilde{\omega}} -(\rho_{jk}+\eta_{jk})\frac{1-\cos(\tilde{\omega})-i\sin(\tilde{\omega})}{i \tilde{\omega}} \right] \nonumber \\
    =& 2\rho_{jk}\left( \frac{\sin \tilde{\omega}}{\tilde{\omega}}\mathcal{C}_{jk}  + \frac{1-\cos\tilde{\omega}}{\tilde{\omega}} \mathcal{S}_{jk}\right) + 2i\eta_{jk}\left( \frac{1-\cos\tilde{\omega}}{\tilde{\omega}} \mathcal{C}_{jk} - \frac{\sin \tilde{\omega}}{\tilde{\omega}} \mathcal{S}_{jk}\right)
\end{align}
\noindent For the third part, we have
\begin{align*}
    \mathbb{III} &\equiv \int_0^1 \int_0^1 e^{i\tilde{\omega}(x-y)} w_{jk}(x-y) |x-y|^{H_j+H_k} \dd x\, \dd y= \quad |\textrm{substitute }u = x-y,\ z = x|\\
    &= \int_{-1}^0 du \int_0^{u+1}dz w_{jk}(u)e^{i\tilde{\omega} u} |u|^{H_j+H_k} + \int_0^1du\int_u^1 dz w_{jk}(u)e^{i \tilde{\omega} u} u^{H_j+H_k}\\
    &= \int_{-1}^0 du w_{jk}(u)e^{i\tilde{\omega} u}|u|^{H_j+H_k}(1+u)+\int_0^1 du w_{jk}(u)e^{i\tilde{\omega} u}u^{H_j+H_k}(1-u)\\
    &= \int_0^1du \underbrace{w_{jk}(-u)}_{=\rho_{jk}+\eta_{jk}}e^{-i\tilde{\omega} u}u^{H_j+H_k}(1-u)+\int_0^1 du \underbrace{w_{jk}(u)}_{=\rho_{jk}-\eta_{jk}}e^{i\tilde{\omega} u}u^{H_j+H_k}(1-u)\\
    &=\rho_{jk}\int_0^1 \underbrace{\left(e^{-i\tilde{\omega} u} + e^{i \tilde{\omega} u}\right)}_{=2\cos(\tilde{\omega} u)} u^{H_j+H_k}(1-u)du - \eta_{jk}\int_0^1 \underbrace{\left(e^{i \tilde{\omega} u}-e^{-i\tilde{\omega} u} \right)}_{=2i\sin(\tilde{\omega} u)} u^{H_j+H_k}(1-u)du \\
    & = 2\rho_{jk}\left( \mathcal{C}_{jk} -\int_0^1 \cos(\tilde{\omega} u) u^{H_j+H_k+1}du\right)-2i\eta_{jk}\left( \mathcal{S}_{jk} -\int_0^1 \sin(\tilde{\omega} u)u^{H_j+H_k+1}du\right).
\end{align*}

\noindent We see that
\begin{align}
    \int_0^1 \cos(\tilde{\omega} u)u^{H_j+H_k+1}du&= \frac{\dd}{\dd\tilde{\omega}}\int_0^1\sin(\tilde{\omega} u)u^{H_j+H_k}du= \frac{\dd}{\dd\tilde{\omega}}\mathcal{S}_{jk},\\
    \int_0^1 \sin(\tilde{\omega} u)u^{H_j+H_k+1}du&= -\frac{\dd}{\dd\tilde{\omega}}\int_0^1\cos(\tilde{\omega} u)u^{H_j+H_k}du= -\frac{\dd}{\dd\tilde{\omega}}\mathcal{C}_{jk},
\end{align}
or, via integration by parts we obtain
\begin{align}
    \int_0^1 \cos(\tilde{\omega} u)u^{H_j+H_k+1}du= \frac{\sin\tilde{\omega}}{\tilde{\omega}}-\frac{H_j+H_k+1}{\tilde{\omega}}\mathcal{S}_{jk}, \\
    \int_0^1 \sin(\tilde{\omega} u)u^{H_j+H_k+1}du= \frac{-\cos\tilde{\omega}}{\tilde{\omega}}+\frac{H_j+H_k+1}{\tilde{\omega}}\mathcal{C}_{jk}.
\end{align}
We then arrive at the final expression
\begin{align}
    \mathbb{I}+\mathbb{II}-\mathbb{III} =& 2\rho_{jk} \left(\left[\frac{\sin \tilde{\omega}}{\tilde{\omega}}-1\right]\mathcal{C}_{jk} + \frac{1-\cos\tilde{\omega}}{\tilde{\omega}}\mathcal{S}_{jk} + \frac{\dd}{\dd\tilde{\omega}} \mathcal{S}_{jk}\right)  \nonumber \\
    & + 2i \eta_{jk} \left(\frac{1-\cos\tilde{\omega}}{\tilde{\omega}}\mathcal{C}_{jk} - \left[\frac{\sin\tilde{\omega}}{\tilde{\omega}}+1\right]\mathcal{S}_{jk} - \frac{\dd}{\dd\tilde{\omega}}\mathcal{C}_{jk}\right),
\end{align}
or equivalently
\begin{align}
    \mathbb{I}+\mathbb{II}-\mathbb{III} =& 2\rho_{jk} \left(\left[\frac{\sin \tilde{\omega}}{\tilde{\omega}}-1\right]\mathcal{C}_{jk} + \frac{1-\cos\tilde{\omega} - H_j -H_k - 1}{\tilde{\omega}}\mathcal{S}_{jk} + \frac{\sin\tilde{\omega}}{\tilde{\omega}} \right) \nonumber \\
    & + 2i \eta_{jk} \left(\frac{1-\cos\tilde{\omega} + H_j +H_k + 1}{\tilde{\omega}}\mathcal{C}_{jk} - \left[\frac{\sin\tilde{\omega}}{\tilde{\omega}}+1\right]\mathcal{S}_{jk}\right).
\end{align}
\end{proof}

\section{Proof of \autoref{thm:psd_asymptotics}, asymptotics of ensemble-averaged PSD}
\label{appendix:proof_psd_asymptotics}
\begin{proof}
Per \cite{krapf2019spectral} we have the asymptotic behaviour of functions $\mathcal{C}_{jk}$ and $\mathcal{S}_{jk}$

\begin{align}
    \mathcal{C}_{jk}(\tilde{\omega}) &\simeq \frac{a_{jk}}{\tilde{\omega}^{H_j+H_k+1}} + \frac{\sin \tilde{\omega}}{\tilde{\omega}} + \frac{(H_j+H_k)\cos\tilde{\omega}}{\tilde{\omega}^2} + \frac{(H_j+H_k)(1-H_j-H_k)\sin\tilde{\omega}}{\tilde{\omega}^3} + O\left(\frac{1}{\tilde{\omega}^4}\right),\\
    \mathcal{S}_{jk}(\tilde{\omega}) &\simeq \frac{a_{jk} \tan\left(\frac{\pi(H_j+H_k)}{2}\right)}{\tilde{\omega}^{H_j+H_k+1}} - \frac{\cos \tilde{\omega}}{\tilde{\omega}} + \frac{(H_j+H_k)\sin\tilde{\omega}}{\tilde{\omega}^2} - \frac{(H_j+H_k)(1-H_j-H_k)\cos\tilde{\omega}}{\tilde{\omega}^3} + O\left(\frac{1}{\tilde{\omega}^4}\right),
\end{align}
where $a_{jk} = \Gamma(H_j + H_k) \sin\left(\frac{\pi}{2}(H_j+H_k)\right)$.

\noindent Thus, grouping the elements by their powers, we have (as $\tilde{\omega}\to \infty$)
    \begin{align}
        \Re \langle S_{jk}(\tilde{\omega}, T)\rangle &\simeq T^{H_j+H_k+1}\sigma_j\sigma_k\rho_{jk}\left\{ \left[\frac{1}{\tilde{\omega}^2} - \frac{(H_j+H_k)\sin\tilde{\omega}}{\tilde{\omega}^3} + O\left(\frac{1}{\tilde{\omega}^4}\right) \right] \right. \\
        &+\left.\frac{a_{jk}}{\tilde{\omega}^{H_j+H_k+1}}\left[1 - \frac{(H_j+H_k+\cos\tilde{\omega})\cot\left(\frac{\pi(H_j+H_k)}{2}\right)}{\tilde{\omega}} \right] \right\},\\
        \Im \langle S_{jk}(\tilde{\omega}, T)\rangle &\simeq  T^{H_j+H_k+1}\sigma_j\sigma_k\eta_{jk}\left\{ \left[\frac{2}{\tilde{\omega}^2} - \frac{(H_j+H_k)\cos\tilde{\omega}(3-\cos\tilde{\omega})}{\tilde{\omega}^3} + O\left(\frac{1}{\tilde{\omega}^4}\right) \right] \right.\nonumber \\
        &+ \left.\frac{a_{jk}}{\tilde{\omega}^{H_j+H_k+1}}\left[\cot\left(\frac{\pi(H_j+H_k)}{2}\right)
         - \frac{2+H_j+H_k-\cos\tilde{\omega} + \sin\tilde{\omega} \cot\left(\frac{\pi(H_j+H_k)}{2}\right)}{\tilde{\omega}} \right] \right\}.
    \end{align}
\end{proof}

\section*{References}

\bibliographystyle{unsrtnat}
\bibliography{bibliography.bib}

\end{document}